\documentclass[10pt,a4paper]{article}

 \usepackage[cmex10]{amsmath}
\usepackage{algorithmic}
\usepackage{array}
 \usepackage{mdwmath}
 \usepackage{mdwtab}
 \usepackage{eqparbox}
\usepackage[tight,footnotesize]{subfigure}
\usepackage{tikz}
\usepackage{bm}
 \usepackage{textpos}
\usepackage{amssymb}
\usepackage{amsfonts}

\usepackage{amsmath}
\usepackage{graphicx}
\usepackage{amssymb}
\usepackage{verbatim}
\usepackage{color}

\usepackage{bigints}

\newcounter{CountEx}
{\bf}{\normalfont}
\newtheorem{problem}{\textbf{Problem}}{\bf}{\normalfont}
\newtheorem{example}{\textbf{Example}}[CountEx]{\bf}{\normalfont}
{\bf}{\normalfont}
{\normalfont}{\normalfont}
{\normalfont}{\normalfont}
\newtheorem{proposition}{\textbf{Proposition}}{\bf}{\normalfont}
\newenvironment*{proof}[1]{\textbf{\emph{Proof}} }{}

\newcommand{\setpoly}[2]{\mathbb{R}_{#1}^{#2}}
\newcommand{\setinteger}[2]{\mathcal{A}_{#1}^{#2}}
\newcommand{\canbas}[2]{b_{#1}^{#2}}

\newcommand{\DP}[1]{{\color{black}{#1}}} %

\newcommand{\Zvec}{z}
\newcommand{\Zetavec}{z}
\newcommand{\Zvecsupp}{\bm{Z}}

\definecolor{comm}{rgb}{0,0,0.9}

\newtheorem{property}{\noindent \textbf{Property}} {\normalfont }{\normalfont}
 {\normalfont }{\normalfont}
 {\normalfont }{\normalfont}
{\normalfont}{\normalfont}
{\normalfont}{\normalfont}

\newtheorem{theorem}{\noindent \textbf{Theorem}}{\normalfont}{\normalfont}
{\normalfont}{\normalfont}
\newtheorem{corollary}{Corollary}{\normalfont}{\normalfont}
\newtheorem{lemma}{Lemma}{\normalfont}{\normalfont}
{\normalfont}{\normalfont}
\newtheorem{remark}{\noindent \textbf{Remark}}{\normalfont}{\normalfont}
{\normalfont}{\normalfont}

\newcommand{\DPrev}[1]{{\color{black}{#1}}} %

\newcommand{\Unc}{\Delta}

\newcommand{\MatrixA}{A}
\newcommand{\sizevec}[1]{n_{#1}}
\newcommand{\unc}{\rho} 

\title{\Large A unified framework for deterministic and probabilistic $\mathcal{D}$-stability analysis of uncertain polynomial matrices \\
\ \\
{\large Technical Report TR-IDSIA-2017-01}
}
\date{}
\author{Dario Piga and Alessio Benavoli}
\begin{document}

\maketitle


\section*{Foreword}
This report is an extended version of the paper \emph{A unified framework for deterministic and probabilistic $\mathcal{D}$-stability analysis of uncertain polynomial matrices} submitted by the authors to the IEEE Transactions on Automatic Control.

\section*{Abstract}
\DPrev{In control theory, we are often interested in robust $\mathcal{D}$-stability analysis, which aims at verifying if all the eigenvalues of an uncertain matrix lie in a given region $\mathcal{D}$ of the complex plane. 
Although many algorithms have been developed to provide conditions  for an uncertain matrix to be robustly $\mathcal{D}$-stable, the problem of computing the probability of an uncertain matrix to be $\mathcal{D}$-stable is still unexplored.   
The goal of this paper is to fill this gap by generalizing algorithms for robust  $\mathcal{D}$-stability analysis in two directions. First, the only constraint on the stability region  $\mathcal{D}$ that we impose is that its complement is a semialgebraic set described by polynomial constraints. This comprises  main important cases in robust control theory.
Second, the  $\mathcal{D}$-stability analysis problem is formulated in a probabilistic framework, by assuming that only few probabilistic information
is  available on the uncertain parameters, such as support and some moments.
 We will  show how to efficiently compute the minimum probability that the matrix is $\mathcal{D}$-stable by using  convex relaxations  based on the theory of moments. We will also show that standard robust $\mathcal{D}$-stability is a particular case of the more general probabilistic $\mathcal{D}$-stability problem. Application to  robustness and probabilistic analysis of dynamical  systems is discussed. }

 \DPrev{In control theory, we are often interested in robust $\mathcal{D}$-stability analysis, which aims at verifying if all the eigenvalues of an uncertain matrix lie in a given region $\mathcal{D}$ of the complex plane. 
Although many algorithms have been developed to provide conditions  for an uncertain matrix to be robustly $\mathcal{D}$-stable, the problem of computing the probability of an uncertain matrix to be $\mathcal{D}$-stable is still unexplored.   
The goal of this paper is to fill this gap by generalizing algorithms for robust  $\mathcal{D}$-stability analysis in two directions. First, the only constraint on the stability region  $\mathcal{D}$ that we impose is that its complement is a semialgebraic set described by polynomial constraints. This comprises  main important cases in robust control theory.
Second, the  $\mathcal{D}$-stability analysis problem is formulated in a probabilistic framework, by assuming that only few probabilistic information
is  available on the uncertain parameters, such as support and some moments.
 We will  show how to efficiently compute the minimum probability that the matrix is $\mathcal{D}$-stable by using  convex relaxations  based on the theory of moments. We will also show that standard robust $\mathcal{D}$-stability is a particular case of the more general probabilistic $\mathcal{D}$-stability problem. Application to  robustness and probabilistic analysis of dynamical  systems is discussed. }

\section{Introduction} 
\DPrev{
\subsection{Motivations}
Consider a plant described by the transfer function
$$
G(s)=\dfrac{\rho_2}{s+\rho_1},
$$
where $\rho_1,\rho_2$ are uncertain parameters belonging to the intervals
 $\rho_1\in [0.035,0.085]$ and $\rho_2\in [12,28]$. 
 Although  these parameters can take any value  in the corresponding uncertainty intervals, we assume that they are usually close to their nominal values (in this case, the centers of the   intervals).  The goal is to design a controller that robustly  stabilizes the closed-loop system and has a fast
 unit step response possibly without overshoots.
%
 Assume we have designed two controllers $\mathcal{K}_{\mathrm{ROB}}$ and $\mathcal{K}_{\mathrm{PROB}}$ and we want to understand which one is better. First we check the stability requirements. Both of them robustly stabilize the closed-loop system. 
 Then we check the performance by evaluating the unit step response of the controlled system for $100$ different parameter realizations  (see Fig. \ref{fig:outputintro}).  It is evident that $\mathcal{K}_{\mathrm{PROB}}$ (red) is preferable in terms of speed (half settling time). However, $\mathcal{K}_{\mathrm{ROB}}$ (blue) has never overshoots, while $\mathcal{K}_{\mathrm{PROB}}$ has one   overshoot out of $100$ parameter realizations. Should we then choose $\mathcal{K}_{\mathrm{ROB}}$ or $\mathcal{K}_{\mathrm{PROB}}$?
 
 Assume we were able to translate the knowledge  that the  parameters $\rho_1,\rho_2$ are usually close to their nominal values in terms of weak probabilistic constraints,  such as $\mathbb{E}[\rho_1]=0.06$, $\mathbb{E}[\rho_2]=20$   and $\mathbb{E}[(\rho_1-0.06)^2]\leq \sigma_1^2= 0.0025$, $\mathbb{E}[(\rho_2-20)^2]\leq \sigma_2^2=0.25$ for instance. Moreover assume that, based on this information,  we could compute the probability that the step response of the closed-loop system  does not have overshoots  and that this probability is  $p=1$ with the controller $\mathcal{K}_{\mathrm{ROB}}$ and  $p=0.95$ with $\mathcal{K}_{\mathrm{PROB}}$. Then, if we accept a $5\%$ probability of having  an overshoot,   we could claim that   $\mathcal{K}_{\mathrm{PROB}}$ is preferable. 
  It is like in car racing:  pilots know that if they exceed the track limits (overshoot) this  causes them to lose time.
 However, they accept a small probability of exceeding the  limits because in this way they can  set faster laps. In this paper, we provide   a mathematical framework to quantify this probability.

\begin{figure}[!b]
\centering
 \includegraphics[ trim={0cm 0cm 9cm 25cm},clip]{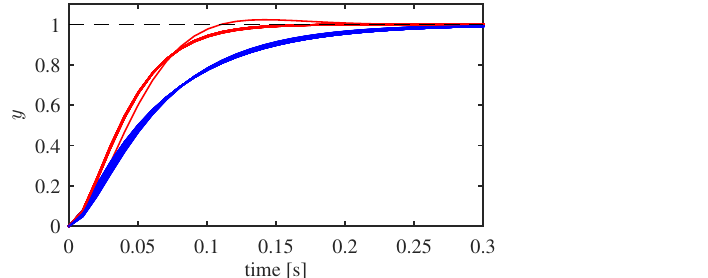}
	\caption{\DPrev{Closed-loop step response with controller $\mathcal{K}_{\mathrm{ROB}}$ (blue) and $\mathcal{K}_{\mathrm{PROB}}$ (red)  for $100$ different realization of the uncertain parameters.} } \label{fig:outputintro}
\end{figure} }

\subsection{Contribution}

\DPrev{
In control theory for \emph{Linear Time-Invariant} (LTI) systems, robust (or probabilistic)  stability and performance requirements can be formulated in terms of robust  (resp., probabilistic)  $\mathcal{D}$-stability analysis, which aims at verifying if (resp., compute the probability that) all the eigenvalues of an uncertain matrix lie in a given region $\mathcal{D}$ of the complex plane.  In this paper, we present   a unified  framework to assess robust and probabilistic
 $\mathcal{D}$-stability of uncertain matrices. Specifically, the contribution of the paper is twofold:}
\begin{enumerate}
\item a novel approach for analysing robust $\mathcal{D}$-stability  of an uncertain matrix $A(\unc)$ is proposed.     The entries of the matrix $A(\unc)$   depend polynomially on an uncertain parameter vector $\unc$, which is assumed to take values in a closed semialgebraic set $\Unc$ described by polynomial constraints.  \DPrev{The only assumption on the stability region $\mathcal{D}$ is that its complement  is a  semialgebraic set (not necessarily convex), described by  polynomial constraints in the complex plane. The addressed problem is quite general and it includes, among others, the analysis of robust nonsingularity,  Hurwitz or Schur stability of a family of matrices with interval, polytopic or $2$-norm bounded perturbations.
}

\item the   $\mathcal{D}$-stability  analysis problem is formulated in a probabilistic framework, by assuming that the uncertain parameters $\unc$ are described by a set of non a-priori specified probability measures. Only the support and some  moments (e.g.,  mean and variance) of the probability measures characterizing the uncertainty $\unc$ are assumed to be known.
This is an approach to robustness, based on    \emph{coherent lower previsions} \cite{walley1991} (also referred to as  \emph{Imprecise Probability})  that has been recently developed in filtering theory \cite{benavoli_2011_journal_c,benavoli2013b,benavoli_2016a}.  Specifically, we seek the ``worst-case probabilistic scenario'', which requires to compute, among all possible  probability measures satisfying the assumptions, the smallest probability of the uncertain matrix $A(\unc)$ to be    $\mathcal{D}$-stable. 
\end{enumerate}

 The latter result  allows us to take into account not only the information about the range of the uncertain parameter $\unc$ (i.e.,  $\unc \in \Unc$),
but also information such as: (1) the nominal value of $\rho$ (e.g., the center of the uncertainty set $\Unc$); (2) the variability of $\rho$ w.r.t.\ its nominal value and so on.
This information is not taken into account in  standard approaches for $\mathcal{D}$-stability analysis, but it allows us to reduce the conservativeness of the obtained results, at the price of guaranteeing $\mathcal{D}$-stability within a given level of probability. 
We can for instance determine if a family of matrices is $\mathcal{D}$-stable with probability $0.90$ or $0.95$ or $0.99$, etc..

To this end, we develop a unified framework for  deterministic (robust) and  probabilistic  $\mathcal{D}$-stability analysis. A semi-infinite linear program is formulated and then relaxed, by exploiting the Lasserre's hierarchy \cite{2001Ala}, into a sequence of (convex) \emph{semidefinite programming} (SDP) problems of finite size.


\subsection{Related works} 
%

 Evaluating the properties of the eigenvalues of a family of  matrices (e.g., robust nonsingularity, maximum real part of the eigenvalues or spectral radius) is an NP-hard problem \cite{PoRo93,Ne93,gurvits2009np}, and there is a vast literature addressing this   research topic. 

 Algorithms for checking  Hurwitz and
Schur stability of  symmetric interval matrices  are proposed in \cite{rohn1989systems,hertz1992extreme,rohn1994positive}, where  it is shown that testing the nonsingularity of symmetric interval matrices requires to calculate  a finite number of determinants, and this number grows  exponentially in the matrix size, limiting the applicability of these methods   to  small scale problems. A branch and bound  algorithm is then proposed in \cite{rohn1996algorithm} to solve larger scale problems. Interval and polytopic matrices are considered in \cite{deif1990adv,qiu1996bounds,alamo2008new,dzhafarov2006stability,peaucelle2000new,bliman2004convex,leite2003improved,chesi2005establishing}. The works in \cite{deif1990adv,qiu1996bounds} derive intervals  where the real eigenvalues of   interval matrices are guaranteed to belong to, \DP{and a vertex result is presented in \cite{alamo2008new} to reduce the computational load in evaluating quadratic stability  of interval matrices. Results in \cite{alamo2008new} can be also used in the case of  multiaffine interval matrix uncertainty.}   Bernstein expansion is used in \cite{dzhafarov2006stability} to check robust nonsingularity of a polytope of real matrices, and 
 sufficient LMI  conditions coming from the Lyapunov theory are derived in \cite{peaucelle2000new,bliman2004convex,leite2003improved,chesi2005establishing} for checking  robust  Hurwitz and
Schur stability   of  matrices with polytopic uncertainty. In \cite{ramos2001less,henrion2003positive,ebihara2005robust,oliveira2007parameter}, less conservative LMI conditions to check robust $\mathcal{D}$-stability  of uncertain polynomial matrices are derived.
 A method based on the \emph{structured
singular value} and on its variant, the \emph{skewed structured singular
value}, is proposed in \cite{KiBr2016} to  analyse the spectrum of uncertain matrices expressed in a \emph{linear fractional representation}.    Numerically efficient algorithms for computing (lower bounds of) the extreme points (e.g., maximum real part and maximum modulus) of the    $\varepsilon$-pseudospectrum of a matrix $A$ are proposed in \cite{guglielmi2011fast} and \cite{guglielmi2013low}, where the $\varepsilon$-pseudospectrum of a given matrix $A$ is defined as the set of   the eigenvalues of the perturbed matrix $A+E$, for all $\|E\|\leq \varepsilon$. Both the Frobenius  and $2$-norm are used to measure the ``amplitude'' of the  perturbation  $E$, and structured perturbations can be also handled. Since lower bounds on the maximum real part and on the maximum modulus of the  $\varepsilon$-pseudospectrum are computed, necessary conditions for robust Hurwitz and
Schur stability of the uncertain matrix $A+E$ can be  derived. NP-hard robust matrix analysis problems are tackled in \cite{vidyasagar2001probabilistic,alamo2009randomized,tempo2012randomized} through \emph{randomized algorithms}, which run in polynomial  time, at the price of providing an erroneous answer (specifically, a false positive) with some probability.   Other contributions addressing robust $\mathcal{D}$-stability analysis, with applications in systems and control theory, can be found in \cite{Barmish94,blondel2000survey,rump2006eigenvalues,vlassis2014polytopic,freitag2014calculating,freitag2014new,petersen2014robust,rostami2015new} and reference therein. 

The list of reviewed works is   far from being exhaustive, but it points out the efforts made by researchers in the last decades to develop methodologies that, in many cases, can be applied to tackle specific  robust $\mathcal{D}$-stability analysis  problems  (e.g., robust nonsingularity,  Hurwitz or Schur stability)   under specific assumptions on the structure of the uncertainty (e.g., interval, polytopic, or 2-norm bounded uncertainty).

\DPrev{In the context of the present paper, it is worth mentioning the works \cite{henrion2012inner} and \cite{hess2016semidefinite}, where two approaches   based on  Lasserre's hierarchy  are proposed  to approximate the stability region of univariate polynomials with uncertain coefficients. These results can be also used   to assess  robust stability of uncertain polynomial matrices. However, unlike the method proposed in this paper, \cite{henrion2012inner} and \cite{hess2016semidefinite} do not consider  the probabilistic scenario  and  a restricted subset of stability regions $\mathcal{D}$ can be handled (for instance, to the best of the authors' knowledge, the complex plane without the imaginary axis cannot be considered as a stability set). Furthermore, a deep Lasserre's hierarchy may be required in \cite{henrion2012inner} and \cite{hess2016semidefinite} to achieve non-conservative results (as discussed in the example reported  in Section \ref{subsec:Henrion}).} 

\subsection{Paper organization}
The paper is organized as follows. The notation  used throughout the paper  is introduced in Section~\ref{Sec:notation}. The   $\mathcal{D}$-stability analysis problem is formally defined in Section~\ref{Sec:PS}, and a unified framework for  deterministic and  probabilistic analysis is provided. The main theorems and results are reported in Section~\ref{sec:Singularity}, where it is shown that the $\mathcal{D}$-stability analysis problem can be formulated as a semi-infinite linear program. Convex relaxation techniques based on the Lasserre's hierarchy~\cite{2001Ala} and aiming at computing the solution of the formulated semi-infinite linear program  are described in Section~\ref{Sec:LMIrelaxation}. Applications of the proposed method are discussed in Section~\ref{Sec:examples}, along  with simulation examples \DPrev{and a comparison with existing approaches for robust $\mathcal{D}$-stability analysis.}     A simple running example is also used throughout the whole paper for illustrative purposes.

\section{Notation} \label{Sec:notation}
Let  us denote with $x_{\textrm{re}}$ and $x_{\textrm{im}}$  the real and imaginary part, respectively, of a complex vector $x$.  
 Let  $z_i$  be the $i$-th component of a vector $z \in \mathbb{R}^{\sizevec{z}}$.  Let $\mathbb{N}$ be the set of natural numbers and $\mathbb{N}_0^{\sizevec{z}}$    the set of $\sizevec{z}$-dimensional vectors with non-negative integer components.

For a given integer $\tau$,   $\setinteger{\tau}{\sizevec{z}}$ is the set defined as $ \left\{\alpha \in \mathbb{N}_0^{\sizevec{z}}: \sum_{i=1}^{\sizevec{z}}\alpha_i \leq \tau \right\}$. 
We will use  the shorthand notation $z^{\alpha}$  for $ z^{\alpha}=z_1^{\alpha_1}\!\cdots \!z_{\sizevec{z}}^{\alpha_{\sizevec{z}}}\!\!=\!\!\prod_{i=1}^{\sizevec{z}}\! z_i^{\alpha_i}$.  Let us denote with      $\setpoly{\tau}{}[z]$ the set of real-valued polynomials in the variable $z \!\in\! \mathbb{R}^{\sizevec{z}}$  with   degree      less  than or equal to  $\tau$, and let $\canbas{\tau}{}(z)$  be the canonical basis of $\setpoly{\tau}{}[z]$, i.e.,  $b_{\tau}(z)\!=\!\left\{z^{\alpha}\right\}_{\alpha \in \setinteger{\tau}{\sizevec{z}}}$\!.
Denote with $\{\bm{g}_{\alpha}\}_{\alpha \in \setinteger{\tau}{\sizevec{z}}}$  the   coefficients of the polynomial $g\in\setpoly{\tau}{}[z]$ 
  in the canonical basis $\canbas{\tau}{}(z)$, i.e., $ g(z) =  \sum_{\alpha \in \setinteger{\tau}{\sizevec{z}}}\!\bm{g}_{\alpha}\!z^{\alpha}$.  In the case $g$ is an $\sizevec{g}$-dimensional vector of polynomials in $\setpoly{\tau}{}[z]$, we denote with $\bm{g}_{i,\alpha}$ the   coefficients of the polynomial $g_i$ 
  in the  basis $\canbas{\tau}{}(z)$.  Let us denote with $deg(g)$ the degree of the polynomial $g$. 
	
	Let $P_{\Zvec}$ be the cumulative distribution function  
of a Borel probability measure $\Pr_{\Zvec}$ on $\mathbb{R}^{\sizevec{z}}$.  
To understand the relationship between  $\Pr_{\Zvec}$ and $P_{\Zvec}$, we can for instance consider $\mathbb{R}$ and in this case we have that $P_{\Zvec}(\Zvec)=\Pr_{\Zvec}(-\infty,\Zvec]$ -- this definition can easily be extended to $\mathbb{R}^{\sizevec{z}}$.
Because of the equivalence between Borel probability measures and cumulative distributions,
hereafter we will use  interchangeably $\Pr_z$ and $P_z$. For an integer $\tau \geq 0$, let  $ m=\{m_{\alpha}\}_{\alpha \in \setinteger{ \tau}{n}}$ be the sequence of moments of a probability measure $\Pr_\Zvec$   on $\mathbb{R}^{\sizevec{z}}$, i.e.,
$m_{\alpha}=\int \Zvec^{\alpha}dP_\Zvec(\Zvec)$.

\section{Problem setting} \label{Sec:PS}

\subsection{Uncertainty description}

Consider an uncertain square real matrix $\MatrixA(\unc)$ of size $\sizevec{a}$, whose entries depend polynomially on an uncertain parameter vector $\unc \in \mathbb{R}^{\sizevec{\unc}}$. The uncertain vector $\unc $ is assumed to belong to a compact semialgebraic uncertainty set $\Unc$, defined as
\begin{align} \label{eqn:DeltaS}
\Unc=\left\{\unc \in \mathbb{R}^{\sizevec{\unc}}: g_i(\unc) \geq 0, \ \ \ i=1,\ldots,\sizevec{g} \right\},
\end{align} 
where $g_i$ are real-valued polynomial functions of $\rho$. 

\begin{example}
\label{ex:1}
Let us introduce a simple example which will be used throughout the paper for   illustrative purposes.
 Let the uncertain  matrix be
\begin{equation}
\label{eq:matex}
 A(\unc) =\left[ \begin{matrix}
  \unc-1 & 0\\
  0 & -1
 \end{matrix}\right],
\end{equation}
with $\unc \in \Unc=[0,1]$. {According to the notation in \eqref{eqn:DeltaS}, the set $\Unc$ is written as:
\begin{equation*}
\Unc=\left\{\unc \in \mathbb{R}: g_1(\unc)\doteq \rho \geq 0, \ \ \ g_2(\unc)\doteq 1-\rho \geq 0 \right\}.
\end{equation*}} \hfill $\blacksquare$
\end{example}

We also assume to have some probabilistic information on the uncertain vector  $\unc$. 
 Specifically, given $\sizevec{f}$ real-valued polynomial functions $f_i$ ($i=1,\ldots,{\sizevec{f}}$) called \textit{generalized polynomial moment functions} (gpmfs) and defined on   $\Unc$, we assume that the  probabilistic information on the    vector  $\unc$ is represented by the expectations of the gpmfs $f_i$, i.e.,
\begin{align} \label{eqn:infogpmfs}
\mathbb{E}\left[f_i\right]=&\int_{\Unc}f_i(\unc)dP_{\unc}(\unc)=\mu_i, \ \  i=1,\ldots,\sizevec{f},
\end{align}
where the integral is a Lebesgue-Stieltjes integral with respect \DP{to} the cumulative distribution function $P_{\unc}$
of a Borel probability measure $\Pr_{\unc}$ on $\Delta$\footnote{The sample space is $\mathbb{R}^{\sizevec{\unc}}$ and we are considering the Borel $\sigma$-algebra. $\Delta$ is assumed to be an element of the $\sigma$-algebra.} and $\mu_i \in \mathbb{R}$ are  finite and known.\footnote{ {Although equality constraints on  the gpmfs $f_i$ are considered in \eqref{eqn:infogpmfs}, the methodology discussed in the paper can   also be used in the  case of inequality constraints.}}

We will always assume that $f_1(\unc)=1$ and since
${\Pr}_{\unc}$ is a probability measure it follows that $\mu_1=1$. In other words, we have 
\begin{equation*} \label{eqn:normP}
\mathbb{E}\left[f_1\right]=\int_{\Delta}dP_{\unc}(\unc)=1,
\end{equation*}
which expresses the fact that ${\Pr}_{\unc}$ is a probability measure with support on $\Delta$:
\begin{equation*}
{\Pr}_{\unc}(\unc \in \Unc) =\int_{\Unc}dP_{\unc}(\unc)=1.
\end{equation*}

Note that the knowledge of the expectation of $\sizevec{f}$ gpmfs $f_i$ is not enough to uniquely define the measures of probability ${\Pr}_{\unc}$, thus we consider the set of all probability measures ${\Pr}_{\unc}$ which are compatible with the information in \eqref{eqn:infogpmfs}: 
\begin{equation} \label{eqn:infogpmfsvec}
\mathcal{P}_{\unc}=\left\{{P}_{\unc}: \int_{\Unc}f_i(\unc)dP_{\unc}(\unc)=\mu_i, ~~i=1,\ldots,\sizevec{f}\right\}.
\end{equation}
With some \DP{abuse} of terminology,   when in the rest of the paper we state that the probability measures ${\Pr}_{\unc}$ belong to $\mathcal{P}_{\unc}$, we actually mean that the corresponding cumulative distribution functions ${P}_{\unc}$ belong to $\mathcal{P}_{\unc}$.

\begin{example}
\label{ex:2}
Let us continue the running Example \ref{ex:1}.
We consider two cases.
\begin{enumerate}
 \item In the first case, the probabilistic information about $\unc$
is expressed by the set of probability measures:
\begin{equation}
 \label{eq:probex1}
\mathcal{P}^{(1)}_{\unc}=\left\{{P}_{\unc}: \int_{0}^1 dP_{\unc}(\unc)=1\right\}.
\end{equation} 
This means that   only   the support $\Unc=[0,1]$ of the probability measures ${P}_{\unc}$ is known.
\item In the second case, the probabilistic information about $\unc$ is expressed by:
\begin{equation}
 \label{eq:probex2}
\mathcal{P}^{(2)}_{\unc}\!=\!\left\{\!{P}_{\unc}: \int_{0}^1 dP_{\unc}(\unc)=1, ~\int_{0}^1 f_2(\unc)\, dP_{\unc}(\unc)=0.5\!\right\}\!\!,
\end{equation} 
with $f_2(\unc)=\unc$. This means that  both the support  and the first moment (i.e., the mean assumed to be $0.5$) of the probability measures ${\Pr}_{\unc}$ are known.
The value of the mean equal to $0.5$ can be interpreted as a knowledge on the nominal value of $\unc$ in the interval $[0,1]$ and, on average, we expect $\unc$ to be equal to $0.5$. Note that we may also assume that other moments of $\unc$  are known; for instance
we may know the variability of $\unc$ w.r.t.\ the mean $0.5$ (i.e., the variance). This case will be considered in the examples reported in Section \ref{Sec:examples}. 
\end{enumerate}
\hfill $\blacksquare$
\end{example}

The  problems reported in the next paragraphs are addressed in this work.

\subsection{Probabilistic $\mathcal{D}$-stability analysis}
A matrix is $\mathcal{D}$-stable if all the eigenvalues belong to a given region $\mathcal{D}$. 
In this paper, we assume that the stability region $\mathcal{D}$  is  an (open) subset  of the complex plane,  whose  complement $\mathcal{D}^c=\mathbb{C} \setminus \mathcal{D}$ (instability region) is  a closed  semialgebraic set described by
\begin{align} \label{eqn:diskcompl}
\mathcal{D}^c=&\left\{\!\lambda \in \mathbb{C}: \! \lambda=\lambda_{\mathrm{re}}+j\lambda_{\mathrm{im}}, \ \ \lambda_{\mathrm{re}}, \lambda_{\mathrm{im}} \! \in \! \mathbb{R}, \right. \\
& \left.  \ d_i(\lambda_{\mathrm{re}},\lambda_{\mathrm{im}})\geq 0,  i=1,\ldots,\sizevec{d} \! \right\}\!, \nonumber 
\end{align} 
with $d_i$ being  real-valued polynomials in the real variables $\lambda_{\mathrm{re}}$ and $\lambda_{\mathrm{im}}$. Note that  $\mathcal{D}$ can be, for instance, the open left half plane, the  unitary disk centered in the origin, or the complex plane without the imaginary axis. 
Therefore, this assumption cover all important cases in stability analysis.

Among all the probability measures belonging to $\mathcal{P}_{\unc}$, we want to find the ``worst-case scenario'' given by the measure of probability ${\Pr}_{\unc}$ which provides the lower  probability that $\MatrixA(\unc)$ has all the eigenvalues in $\mathcal{D}$ or, equivalently, 
the upper  probability $\overline{p}=1-\underline{p}$ that  $\MatrixA(\unc)$ has  at least an eigenvalue in  $\mathcal{D}^c$.
In this way, we can claim that the probability of the matrix $\MatrixA(\unc)$ to be $\mathcal{D}$-stable w.r.t. the uncertainties  $\unc$  is greater than or equal to $\underline{p}$ (equiv. $1-\overline{p}$). 


Formally, we are interested in solving the following eigenvalue location problem.

%
%

\begin{problem}\textbf{[Probabilistic eigenvalue violation]} \label{Prob:eigviolate}\\
Given the uncertain matrix $\MatrixA(\unc)$,  the uncertain parameter vector $\unc$ with (unknown) measure of probability $\Pr_{\unc}$ belonging  to $\mathcal{P}_{\unc}$, and  a stability region $\mathcal{D}$, compute 
\begin{align} \label{eqn:prob2v2}
\overline{p}=\sup_{P_{\unc} \in \mathcal{P}_{\unc}} {\Pr}_{\unc}\left(\Lambda(\MatrixA(\unc)) \nsubseteq \mathcal{D}\right), 
\end{align}
where $\Lambda(\MatrixA(\unc)) $ is the spectrum of the matrix $\MatrixA(\unc)$, or equivalently, 
\begin{align} \label{eqn:prob2v3}
\overline{p}_{}\!=\!\sup_{P_{\unc} \in \mathcal{P}_{\unc}} {\Pr}_{\unc}\left(\lambda_i(\MatrixA(\unc)) \in \mathcal{D}^c\right), \ \mathrm{for \ some \ } i=1,\ldots,\sizevec{a}.
\end{align}
\hfill $\blacksquare$ 
\end{problem}


\begin{example}
\label{ex:3}
Let us again consider the running example.
As a stability region, we consider the open left half-plane
$$
\mathcal{D}=\left\{\lambda \in \mathbb{C} \mid \lambda_{re}< 0 \right\},
$$
whose complement is the semi-algebraic set:
{$$
\mathcal{D}^c=\left\{\lambda \in \mathbb{C} \mid d_1(\lambda_{re})\doteq\lambda_{re}\geq 0 \right\}.
$$}
Since the eigenvalues of the matrix $A(\unc)$ in (\ref{eq:matex}) are $-1$ and $\unc-1$,
the only eigenvalue that can lead to instability is $\unc-1$.
Therefore, in this case the problem   (\ref{eqn:prob2v3}) becomes:
\begin{align} \label{eqn:prob2v3exx}
\overline{p}_{}\!=\!\sup_{P_{\unc} \in \mathcal{P}_{\unc}} {\Pr}_{\unc}\left(\unc-1\geq 0\right),
\end{align}
where we have exploited the fact that $\mathcal{D}^c=\left\{\lambda \in \mathbb{C} \mid \lambda_{re}\geq 0 \right\}$
and $\lambda_{re}=\unc-1$.
Thus, problem  (\ref{eqn:prob2v3}) aims at computing the upper probability that the matrix $A(\unc)$ is not $\mathcal{D}$-stable, given the probabilistic
information on $\unc$ expressed by the set of feasible cumulative distribution functions $\mathcal{P}_{\unc}$. \hfill $\blacksquare$
\end{example}


The following theorem shows that the challenging problem of verifying deterministic (robust)  $\mathcal{D}$-stability of $\MatrixA(\unc)$ is  a
special case of Problem \ref{Prob:eigviolate}.

\begin{theorem}[Deterministic eigenvalue violation] \label{th:DEV}
 In the case the only information on $\unc$ is the support $\Unc$ of the  probability measures $\Pr_{\unc}$ (namely, we only know that 
 $\unc \in \Unc$), the solution $\overline{p}_{}$ of problem \eqref{eqn:prob2v2} can be either $1$ or $0$. Specifically, $\overline{p}_{}=1$ if  $\MatrixA(\unc)$ is not robustly $\mathcal{D}$-stable w.r.t. the uncertainty set $\Unc$,  $\overline{p}_{}=0$ otherwise.
\end{theorem}
\begin{proof}
First of all observe that
$$
{\Pr}_{\unc}\left(\Lambda(\MatrixA(\unc)) \nsubseteq \mathcal{D} \right)=\int_{\Unc} (1-\mathbb{I}_{\mathcal{D}}\left(\Lambda(\MatrixA(\unc)))\right)dP_{\unc}(\unc),  
$$
where $1-\mathbb{I}_{\mathcal{D}}\left(\Lambda(\MatrixA(\unc))\right)$ is the complement of the indicator function:
\begin{align}
\mathbb{I}_{\mathcal{D}}(\Lambda(\MatrixA(\unc)))=\left\{ \begin{array}{ll}
1 & \textrm{if\ } \Lambda(\MatrixA(\unc))  \subseteq \mathcal{D},\\
0 & \textrm{otherwise},
\end{array}\right.
\end{align}
and $P_{\unc} \in \mathcal{P}_{\unc}$ with
\begin{equation} \label{eqn:infogpmfsvecdet}
\mathcal{P}_{\unc}=\left\{{P}_{\unc}: \int_{\Unc} dP_{\unc}(\unc)=1\right\}.
\end{equation}
$\mathcal{P}_{\unc}$ includes all the probability measures supported by $\Unc$ and so it also includes
atomic measures (Dirac's delta) with support in $\Unc$.
Hence, assume that the matrix $\MatrixA(\unc)$ is not robustly $\mathcal{D}$-stable against $\Unc$. Thus, there exists $\hat{\unc} \in \Unc$ such that $\Lambda(\MatrixA(\hat{\unc})) \nsubseteq \mathcal{D}$. 
Then we can take  ${\Pr}_{\unc}$ equal to the Dirac's delta centred on $\hat{\unc}$ and we have that 
$$
{\Pr}_{\unc}\left(\Lambda(\MatrixA(\hat{\unc})) \nsubseteq \mathcal{D}\right)=1.
$$
Similarly, assume that the matrix $\MatrixA(\unc)$ is $\mathcal{D}$-stable for any $\unc \in \Unc$. Then $(1-\mathbb{I}_{\mathcal{D}}\left(\Lambda(\MatrixA(\unc)))\right)=0$ for any $\unc \in \Unc$. Thus,  
${\Pr}_{\unc}\left(\Lambda(\MatrixA(\hat{\unc})) \nsubseteq \mathcal{D} \right)=0$.
\end{proof}
 
\ \\

\begin{example}
\label{ex:4}
Let us go back to our running example assuming the set of probability measures (\ref{eqn:infogpmfsvecex}). 
Theorem \ref{th:DEV} proves that the robust  $\mathcal{D}$-stability analysis problem can   be reformulated 
in a probabilistic way by writing the deterministic constraint $\unc \in \Unc=[0,1]$
as the equivalent probabilistic constraint:
\begin{equation} \label{eqn:infogpmfsvecex}
{P}^{}_{\unc} \in \mathcal{P}^{(1)}_{\unc}=\left\{{P}_{\unc}: \int_{0}^1 dP_{\unc}(\unc)=1\right\}.
\end{equation}
We can then determine the upper probability that the matrix is  unstable 
by solving the optimization problem:
\begin{align} \label{eqn:prob2v3exxx}
\overline{p}_{}\!=\!\sup_{P_{\unc} \in \mathcal{P}^{(1)}_{\unc}} {\Pr}_{\unc}\left(\unc-1\geq 0\right). 
\end{align} 
The solution of the above optimization problem is given by the probability measure $\Pr_{\unc}=\delta_{(1)}$, i.e., an
atomic measure (Dirac's delta) centered at  $\unc=1$. In fact, this measure
belongs to $\mathcal{P}^{(1)}_{\unc}$ since 
$$
\int_{0}^1 \delta_{(1)}(\unc)d\unc=1,
$$
and therefore is compatible with the probabilistic information on $\unc$.
Moreover, for this measure, we have:
 $$
 {\Pr}_{\unc}\left(\unc-1\geq 0\right)= \int_{0}^1 \mathbb{I}_{[1,\infty)}(\unc) \delta_{(1)}(\unc)d\unc=1,
 $$
 where $\mathbb{I}_{[1,\infty)}(\unc)$ is the indicator function of the set $[1,\infty)$.  Since  $\overline{p}_{}=1$,  we can conclude that there exists at least one value of $\unc$ in $\Unc$
such that the matrix is not  $\mathcal{D}$-stable. 
 \hfill $\blacksquare$
\end{example}

 Theorem \ref{th:DEV} shows that the deterministic $\mathcal{D}$-stability analysis problem is a particular case of probabilistic   $\mathcal{D}$-stability analysis.  Although the result in Theorem \ref{th:DEV} is quite intuitive, it is fundamental to formulate, in a rigorous way, the deterministic and the probabilistic $\mathcal{D}$-stability analysis problem in a unified framework.  In fact, one could erroneously   think that the probabilistic constraint equivalent to  $\unc \in \Unc$ is
\begin{equation}
\label{eq:unif}
\int_{\Unc} \frac{1}{|\Unc|} d\unc=1,
\end{equation} 
where $|\Unc|$ is the Lebesgue measure of $\Unc$,  i.e., ${\Pr}_{\unc}$ is equal to the uniform distribution
on $\Unc$. This is not the case as illustrated in the following example.

\begin{example}
\label{ex:4_v2} 
If in the running example we   translate the (deterministic) information   $\unc \in \Unc$ as in \eqref{eq:unif} (with $|\Unc|=1$),  the probability that the matrix is  unstable would be equal to zero,
since the only value that gives instability ($\rho=1$) has zero Lebesgue measure.
The mistake here is that the uniform distribution is just  one of the possible probability measures with support on $\Unc$. There
are infinite of such distributions and, as discussed above, the one that gives rise to instability is an atomic measure on the 
value $\rho=1$. Thus, the equivalent of the constraint $\unc \in \Unc$ is (\ref{eqn:infogpmfsvecex}) and not (\ref{eq:unif}). \hfill $\blacksquare$
\end{example}


\section{A moment problem for $\mathcal{D}$-stability analysis} \label{sec:Singularity}
As shown in Theorem \ref{th:DEV}, the problem of evaluating (deterministic) robust $\mathcal{D}$-stability of an uncertain matrix $\MatrixA(\unc)$ is a particular  case of    probabilistic  $\mathcal{D}$-stability analysis. However, for the sake of exposition,    we first provide results in the deterministic setting, where only the set $\Unc$ where the uncertainty $\unc$ belongs to is assumed to be known. The probabilistic scenario,   where the expectations of the generalized polynomial moment functions $f_i$  of $\unc$ are known (eq. \eqref{eqn:infogpmfs}), will be   discussed later.


\subsection{Checking determinist $\mathcal{D}$-stability}
The following theorem (based on a proper extension of the results recently proposed  by one of the authors in \cite{Pi2016} to compute the \emph{structured singular value} of  a matrix)   provides necessary and sufficient conditions to check determinist (robust) $\mathcal{D}$-stability of the matrix $A(\unc)$ against the uncertainty set   $\Unc$.

\begin{theorem} \label{theor:loctheorem}
All eigenvalues of the  matrix $A(\unc)$ are   located in the set $\mathcal{D}$ for all uncertainties $\unc \in \Unc$ if and only if  
 the solution of the following (nonconvex) optimization problem is $0$: 
\begin{subequations} \label{eqn1:theloc}
\begin{align}
   & \max_{\begin{array}{l}
               x \in \mathbb{C}^{\sizevec{a}},\unc \in \Unc, \lambda \in \mathbb{C}
             \end{array}
}   \|x\|_2^2\\ \vspace*{-0.2cm}
  &s.t.    \ \ \left(A(\unc)-\lambda I\right)x=0, \ \  \|x\|^2 \leq 1, \ \ \lambda \in \mathcal{D}^c. 
\end{align}
\end{subequations} 
\end{theorem}

\noindent \begin{proof}\ \  
 First, the  ``only if'' part is proven. If all the eigenvalues of $A(\unc)$ are located in the set $\mathcal{D}$ (or equivalently, no eigenvalue of $A(\unc)$ belongs to the complement set $ \mathcal{D}^c$), there exists no  value $\lambda \in \mathcal{D}^c$ and $\unc \in \Unc$ which make the matrix $A(\unc)-\lambda I$ singular. Thus,  only the trivial solution $x=0$ satisfies the constraint $\left(A(\unc)-\lambda I\right)x=0$. Therefore, the solution of problem \eqref{eqn1:theloc} is equal to zero. 

The ``if'' part is proven by contradiction. Assume   there exists an uncertainty $\unc \in \Unc$ such that  an eigenvalue $\lambda_i$   of $A(\unc)$ belongs to $\mathcal{D}^c$. Thus, the corresponding eigenvector $x^* \neq 0$ satisfies the constraint  $\displaystyle \left(A(\unc)-\lambda_i I\right)x^*=0$. 
Furthermore, for any $\beta \in \mathbb{C}$, also $x=\beta x^*$  satisfies the constraint  $\displaystyle \left(A(\unc)-\lambda_i I\right)x=0$. 
Thus, the supremum of the $2$ norm of the set of vectors $x$  satisfying $\displaystyle \left(A(\unc)-\lambda_i I\right)x=0$ is infinity. Since  the constraint  $\|x\|^2 \leq 1$ is present in   \eqref{eqn1:theloc}, the   solution of problem \eqref{eqn1:theloc} is $1$, contradicting the hypothesis.
\end{proof}
\ \\

\begin{corollary} \label{theor:loctheorem2}
There exists an uncertainty $\unc \in \Unc$ such that at least an eigenvalue $\lambda_i$ of $\MatrixA(\unc)$ does not belong to $\mathcal{D}$ if and only if the solution of the problem  \eqref{eqn1:theloc} is $1$.
\end{corollary}
\noindent \begin{proof}\ \
It follows straightforwardly from Theorem \ref{theor:loctheorem} and its proof.
\end{proof}

\begin{example}
In the explanatory example considered so far,  problem \eqref{eqn1:theloc} is:
\begin{subequations} \label{eqn1:thelocexv2}
\begin{align}
   & \max_{\begin{array}{l}
               x \in \mathbb{R}^{2}, \unc \in [0 \ \ 1],  \lambda_{\mathrm{re}} \in \mathbb{R}
             \end{array}
}   \|x\|_2^2\\ \vspace*{-0.2cm}
\nonumber
  s.t.  \\
  & {\small   
 \left[\!\!\!\begin{array}{cc}
	\rho\!-\!1\!-\!\lambda_{\mathrm{re}} & 0 \\
	0 & -1 \!-\!\lambda_{\mathrm{re}}
	\end{array}
	\!\!\!\right] 	\!\!\left[\!\!\!\begin{array}{c}
	 x_1 \\
	 x_2
	\end{array}\!\!\!
	\right]\!=\!\left[\!\!\!\begin{array}{c}
	 0 \\
	 0
	\end{array}\!\!\!
	\right], \   \|x\|^2 \! \leq \! 1\!, \  \lambda_{\mathrm{re}} \!\geq \! 0. }
\end{align}
\end{subequations} 
where we have exploited the fact that, since $A(\rho)$ is a real symmetric matrix, its eigenvalues are real.
A feasible  point of problem \eqref{eqn1:thelocexv2} is   $\rho=1$,  $\lambda_{\mathrm{re}}=0$, and $[x_1 \ \ x_2]^\top=[1 \ \ 0]^\top$. At this point,   $\|x\|^2=1$,
which is the maximum of $\|x\|_2^2$ under the constraint $\|x\|_2^2\leq 1$. Thus, according to Theorem \ref{theor:loctheorem} and Corollary \ref{theor:loctheorem2},  the matrix $A(\unc)$ is not robustly $\mathcal{D}$-stable.  \hfill $\blacksquare$
\end{example}

\subsection{Checking probabilistic $\mathcal{D}$-stability}
Let us now focus on the probabilistic $\mathcal{D}$-stability analysis problem, which aims  at  
computing $\overline{p}_{}$, namely, the upper probability among the probability measures  in $\mathcal{P}_{\unc}{(\mu)}$ of the matrix $\MatrixA(\unc)$ to have at least an eigenvalue in the instability region $\mathcal{D}^c$   (see Problem \ref{Prob:eigviolate}). The following theorem, which can be seen as the probabilistic version of Theorem \ref{theor:loctheorem} and Corollary \ref{theor:loctheorem2}, shows how the computation $\overline{p}_{}$ can be formulated as a moment optimization problem.  

\begin{theorem}
Given the uncertain matrix $\MatrixA(\unc)$, the uncertain parameter vector $\unc$ whose measures of probability $\Pr_{\unc}(\unc)$ are constraint to belong to $\mathcal{P}_{\unc}$, and the instability region $\mathcal{D}^c$,  the upper probability $\overline{p}_{}$ (defined in \eqref{eqn:prob2v3}) of the matrix $\MatrixA(\unc)$ to have at least an eigenvalue in $\mathcal{D}^c$ is given by the solution of the following optimization problem:

{\small \begin{subequations} \label{eqn:probProbloc}
\begin{align}  
\overline{p}_{}=& \sup_{P_{\unc,x,\lambda}} \iiint \|x\|^2 dP_{\unc,x,\lambda}(\unc,x,\lambda) \label{conP:1} \\
  \nonumber
  s.t.  \\
	&  \DPrev{\iiint  dP_{\unc,x,\lambda}(\unc,x,\lambda)=1, \label{conP:supp}} \\
			&  \int_{\unc \in \Unc} \int_{\|x\|^2 \leq 1} \int_{\lambda \in \mathcal{D}^c} dP_{\unc,x,\lambda}(\unc,x,\lambda)=1, \label{conP:4} \\
  &  \int_{\unc \in \Unc} \int_{\|x\|^2 \leq 1} \int_{\lambda \in \mathcal{D}^c}  \!\! \!\!f_i(\unc)dP_{\unc,x,\lambda}(\unc,x,\lambda)\!=\!\mu_i, \  i\!=\!2,\ldots,\sizevec{f}\!, \label{conP:2} \\
&  \iiint\limits_{\left(\MatrixA(\unc)-\lambda I\right)x=0}dP_{\unc,x,\lambda}(\unc,x,\lambda)=1, \label{conP:3}  
\end{align} 
\end{subequations}}

with $P_{\unc,x,\lambda}$ being the joint cumulative distribution function of the variables $(\unc,x,\lambda)$. 
\end{theorem}

Observe that  (\ref{conP:4}) is just the moment constraint:
$$
\int_{\unc \in \Unc} \int_{\|x\|^2 \leq 1} \int_{\lambda \in \mathcal{D}^c}  \!\! \!\!f_1(\unc)dP_{\unc,x,\lambda}(\unc,x,\lambda)\!=\!\mu_1,
$$
which has been explicited to highlight  the support of $P_{\unc,x,\lambda}$.

\begin{proof}
First, note that the constraints \DPrev{\eqref{conP:supp} and \eqref{conP:4} guarantee}  that  $P_{\unc,x,\lambda}$ is a cumulative distribution function of a probability  distribution ${\Pr}_{\unc,x,\lambda}$, whose marginals ${\Pr}_{\unc}$, ${\Pr}_{x}$ and ${\Pr}_{\lambda}$ are supported by $\Unc$, $\{x \in \mathbb{C}^{\sizevec{a}}: \|x\|^2 \leq 1\}$, and $\mathcal{D}^c$, respectively. Furthermore, the constraint in \eqref{conP:2} guarantees that $P_{\unc} \in \mathcal{P}_{\unc}$, in fact:
{ \begin{align}
 & \int_{\unc \in \Unc} \int_{\|x\|^2 \leq 1} \int_{\lambda \in \mathcal{D}^c} \!\!\! f_i(\unc)dP_{\unc,x,\lambda}(\unc,x,\lambda) \nonumber \\
= & \int_{\Unc}\!\!\! f_i(\unc)dP_{\unc}(\unc)\!=\!\mu_i,  \ i\!=\!2,\ldots,\sizevec{f}. 
\end{align}} 
 Let us now consider the constraint \eqref{conP:3}. The following two situations may occur:
\begin{enumerate}
\item  the pair $\hat{\unc}$ and $\hat{\lambda}$ does not make the matrix $\MatrixA(\hat{\unc})-\hat{\lambda}I$ singular (namely, $\hat{\lambda}$ is not an eigenvalue of  $\MatrixA(\hat{\unc})$). Then, the only value of $x$ in the integral domain $\left(\MatrixA(\unc)-\lambda I\right)x=0$ is $x=0$. Thus, only a joint cumulative probability distribution $P_{\unc,x,\lambda}$ with marginal probability distribution ${\Pr}_x=\delta_{(0)}(x)$  satisfies \eqref{conP:3}. 
\item  the pair $\hat{\unc}$ and $\hat{\lambda}$ makes the matrix $\MatrixA(\hat{\unc})-\hat{\lambda}I$ singular (namely, $\hat{\lambda}$ is an eigenvalue of  $\MatrixA(\hat{\unc})$).  Thus, any left  eigenvector $\hat{x} \neq 0$ of the matrix $\MatrixA(\hat{\unc})$ associated to the eigenvalue $\hat{\lambda}$ satisfies $\left(\MatrixA(\unc)-\lambda I\right)x=0$. Thus, the marginal $dP_x$ of the joint   $dP_{\unc,x,\lambda}$ is not constraint to have its mass centered in $x=0$.  
It depends on the value of  $\unc,\lambda$, i.e., $P_x(\cdot|\unc,\lambda)$,  we can the decompose    $dP_{\unc,x,\lambda}$ as $dP_{\unc,x,\lambda}=dP_x(\cdot|\unc,\lambda)dP_{\unc,\lambda}$.
\end{enumerate}
Based on the considerations above, the support $S_x(\cdot|\unc,\lambda)$ of the  marginal probability distribution ${\Pr}_x(\cdot|\unc,\lambda)$ is either
\begin{align}
\hspace{-2cm} S_x(\cdot|\unc,\lambda)=& \{0\} \\
  & \textrm{\ if\ } \MatrixA(\unc)-\lambda I \textrm{\ is nonsingular}, \nonumber 
\end{align}
or 
\begin{align}
S_x(\cdot|\unc,\lambda)=&\left\{x: \|x\|^2\leq 1, \ \left(\MatrixA(\unc)-\lambda I\right)x=0 \right\} \label{eqn:Sx} \\
  &  \textrm{ if\ } \MatrixA(\unc)-\lambda I \textrm{\ is singular}. \nonumber 
\end{align}
Let us rewrite the joint  $dP_{\unc,x,\lambda}$ as $dP_{\unc,x,\lambda}=dP_x(\cdot|\unc,\lambda)dP_{\unc,\lambda}$ and let us split the objective function in \eqref{conP:1} 
as:
\begin{subequations} \label{eqn:splittot}
\begin{align}
 &\iiint \|x\|^2 dP_{\unc,x,\lambda}(\unc,x,\lambda)  \\
              =  & \iiint\displaylimits_{\scriptsize \begin{array}{c}A(\unc)-\lambda I \\
							\textrm{nonsingular} \end{array}} \|x\|^2 dP_{x}(x|\unc,\lambda)dP_{\unc,\lambda}(\unc,\lambda) + \label{eqn:split1} \\
               +  & \iiint\displaylimits_{\scriptsize \begin{array}{c}A(\unc)-\lambda I \\
							\textrm{singular} \end{array}} \|x\|^2 dP_{x}(x|\unc,\lambda)dP_{\unc,\lambda}(\unc,\lambda). \label{eqn:split2} 							
\end{align}
\end{subequations}
Let us consider the term \eqref{eqn:split1}. Based on the above considerations, for any   probability measure satisfying the constraints \eqref{conP:4}-\eqref{conP:3}, we have:
\begin{subequations} \label{eqn:objzero}
\begin{align} 
& \iiint\displaylimits_{\scriptsize \begin{array}{c}A(\unc)-\lambda I \\
							\textrm{nonsingular} \end{array}} \! \!\! \!\!\! \|x\|^2 dP_{x}(x|\unc,\lambda)dP_{\unc,\lambda}(\unc,\lambda)   \\
						=  &	\iiint \! \! \|x\|^2 \delta_{(0)}(x)dx dP_{\unc,\lambda}(\unc,\lambda)\!=\!0. 
\end{align}   
\end{subequations} 
Let us consider the term \eqref{eqn:split2}.  For any   probability measure satisfying the constraints \eqref{conP:4}-\eqref{conP:3}, we have:
\begin{subequations}
\begin{align}
& \iiint\displaylimits_{\scriptsize \begin{array}{c}A(\unc)-\lambda I \\
							\textrm{singular} \end{array}} \|x\|^2 dP_{x}(x|\unc,\lambda)dP_{\unc,\lambda}(\unc,\lambda)   \label{eqn:term1}   \\
 \leq &   \iiint\displaylimits_{\scriptsize \begin{array}{c}A(\unc)-\lambda I \\
							\textrm{singular} \end{array}}  dP_{\unc,\lambda}(\unc,\lambda)	= \mathrm{Pr}_{\unc}\left(\Lambda\left(A(\unc)\right) \nsubseteq \mathcal{D}\right)=\overline{p},			\label{eqn:term2} 			
\end{align}
\end{subequations}
where the inequality comes from  the fact that the support of ${\Pr}_{x}(x|\unc,\lambda)$ is bounded by $\|x\|^2 \leq 1$ (see eq.  \eqref{eqn:Sx}). Among all the feasible conditional distributions ${\Pr}_{x}(\cdot|\unc,\lambda)$, which are constrained to have  support $S_x$ in \eqref{eqn:Sx}, let us consider the Dirac's function $\delta_{(\hat{x})}$ centered at $\hat{x}$, with $\hat{x}: \|\hat{x}\|^2=1$. For such a distribution, the term \eqref{eqn:term1} is equal to:
 
\begin{align}  \label{eqn:term3} 
& \iiint\displaylimits_{\scriptsize \begin{array}{c}A(\unc)-\lambda I \\
							\textrm{singular} \end{array}} \|x\|^2 dP_{x}(x|\unc,\lambda)dP_{\unc,\lambda}(\unc,\lambda)  \nonumber \\
							=  & \iiint\displaylimits_{\scriptsize \begin{array}{c}A(\unc)-\lambda I \\
							\textrm{singular} \end{array}} \! \! \|x\|^2 \delta_{(\hat{x})}(x)dx dP_{\unc,\lambda}(\unc,\lambda) \nonumber \\
							=  & \iiint\displaylimits_{\scriptsize \begin{array}{c}A(\unc)-\lambda I \\
							\textrm{singular} \end{array}}  dP_{\unc,\lambda}(\unc,\lambda)	= \mathrm{Pr}_{\unc}\left(\Lambda\left(A(\unc)\right) \nsubseteq \mathcal{D}\right)=\overline{p}.
\end{align} 
Thus, from \eqref{eqn:term3}  and the upper bound  in \eqref{eqn:term2}, we have that, at the optimum, 
 \begin{align} \label{eqn:fintot}
& \iiint\displaylimits_{\scriptsize \begin{array}{c}A(\unc)-\lambda I \\
							\textrm{singular} \end{array}} \!\!\!\! \! \!\!\|x\|^2 dP_{x}(x|\unc,\lambda)dP_{\unc,\lambda}(\!\unc,\lambda\!) \! = \! \mathrm{Pr}_{\unc}\left(\!\Lambda\left(A(\unc)\!\right) \! \nsubseteq \! \mathcal{D}\right)\!=\!\overline{p}.
\end{align}		
By combining 	eq. \eqref{eqn:splittot} with  the conditions \eqref{eqn:objzero} and		\eqref{eqn:fintot}, the theorem follows.
 \end{proof}
\ \\

The intuitive explanation behind the formulation of problem \eqref{eqn:probProbloc} is the following.
 According to Theorem \ref{theor:loctheorem} and Corollary \ref{theor:loctheorem2}, when the optimum of the deterministic problem \eqref{eqn1:theloc} is achieved,   $\|x\|^2=1$ if  $\Lambda\left(\MatrixA(\unc)\right) \nsubseteq  \mathcal{D}$, $0$ otherwise.  
Thus, when the information on $\unc$ is modeled in terms of probability measures,  $\|x\|^2$ becomes a uncertain variable which takes the values:
\begin{align*}
\|x\|^2=\left\{ \begin{array}{ll}
0 & \ \mathrm{if \ } \Lambda\left(\MatrixA(\unc)\right) \subseteq  \mathcal{D},\\ 
1 & \ \mathrm{if \ } \Lambda\left(\MatrixA(\unc)\right) \nsubseteq  \mathcal{D}.
\end{array} \right.
\end{align*}
Thus, the expected value of $\|x\|^2$ (namely, the objective function in \eqref{eqn:probProbloc}) coincides with  ${\Pr}_x(\|x\|^2=1)$, which in turn provides  $\textrm{Pr}_{\unc}\left(\Lambda\left(\MatrixA(\unc)\right) \nsubseteq  \mathcal{D}\right)$. 

The constraints in \eqref{conP:4} and \eqref{conP:3} are simply the ``probabilistic version'' of the determinist constraints in \eqref{eqn1:theloc}, and they are used to describe  the support of the   probability measures ${\Pr}_{\unc,x,\lambda}$.   The constraint \eqref{conP:2} includes the information in \eqref{eqn:infogpmfs} on the (generalized) moments of the probability measures  ${\Pr}_{\unc}$, i.e., ${P}_{\unc} \in \mathcal{P}_{\unc}$.

{\begin{example} \label{ex:Probc2}
Let us continue the explanatory example, and consider the case where the  probabilistic information on $\unc$ is expressed by the set $\mathcal{P}^{(2)}_{\unc}$ (eq. \eqref{eq:probex2}).  Then, problem \eqref{eqn:probProbloc} is given by:
\begin{subequations} \label{eqn:probProblocex4}
\begin{align}  
\overline{p}_{}=& \sup_{P_{\unc,x,\lambda}} \iiint \|x\|^2 dP_{\unc,x,\lambda}(\unc,x,\lambda)  \\
  \nonumber
  s.t. \\
	 & \DPrev{\int  dP_{\unc,x,\lambda}(\unc,x,\lambda)=1, } \\
  			&  \int_{\unc \in [0  \ 1]} \int_{\|x\|^2 \leq 1} \int_{\lambda_{\mathrm{re}} \geq 0} dP_{\unc,x,\lambda}(\unc,x,\lambda)=1,  \label{eq:conss4} \\
  &  \int_{\unc \in [0  \ 1]} \int_{\|x\|^2 \leq 1} \int_{\lambda_{\mathrm{re}} \geq 0}  \unc dP_{\unc,x,\lambda}(\unc,x,\lambda)=0.5, \label{eq:conss2}   \\
	     &  \iiint\limits_{\left(\MatrixA(\unc)-\lambda_{\mathrm{re}} I\right)x=0}dP_{\unc,x,\lambda}(\unc,x,\lambda)=1.  \label{eq:conss3}   
\end{align} 
\end{subequations}
Because of the constraint \eqref{eq:conss4},  the joint    distribution $\Pr_{\unc,x,\lambda}$ is supported by 
\begin{equation*}
\left\{(\unc,x,\lambda_{\mathrm{re}}):  \ \ \unc \in [0 \ 1], \ \|x\|^2 \leq 1, \ \lambda_\mathrm{re}\geq 0 \right\}.
\end{equation*}
We remind that $A(\unc)$ is unstable if and only if $\unc=1$. For this value of $\unc$, $A(\unc)$ has an eigenvalue in zero.
Let us rewrite $P_{\unc,x,\lambda}$ as $P_x(\cdot|\unc,\lambda_\mathrm{re})P_{\unc,\lambda_\mathrm{re}}$.
Then, because of \eqref{eq:conss3}, the conditional marginal distribution $P_x(\cdot|\unc,\lambda_\mathrm{re})$ is supported by:
\begin{align*}
\left\{x\!:\!   \|x_1\|^2 \leq 1, \ \DP{x_2=0} \! \right\} \ & \textrm{if\ } \unc=1 \textrm{\ and\ } \lambda_\mathrm{re}=0, \\
\{0\} \ &   \textrm{if\ } \unc \neq 1  \textrm{\ or\ }\lambda_\mathrm{re} \neq 0.
\end{align*}
Thus, at the optimum, the objective function of problem \eqref{eqn:probProblocex4}  is given by
\begin{equation}  \label{eq:exobjopt}
\int_{\rho=1} \int_{\lambda_\mathrm{re}=0}dP_{\unc,\lambda}(\unc,\lambda).  
\end{equation}
Among all  the probability measures $\Pr_{\unc,\lambda}$  satisfying  the moment constraint \eqref{eq:conss2} on the marginal distribution  $\Pr_{\unc}$ and the constraints \eqref{eq:conss4}-\eqref{eq:conss3}, the one maximizing \eqref{eq:exobjopt} is given by 
\begin{equation} \label{eqn:Prex}
{\Pr}_{\unc,\lambda}(\unc,\lambda)=\left(0.5\delta_{(0)}(\unc)+0.5\delta_{(1)}(\unc)\right)\delta_{(0)}(\lambda_\mathrm{re}).
\end{equation}
Thus, the maximum value of the objective function in \eqref{eq:exobjopt} is given by: 
\begin{align*}
&\int_{\rho=1} \int_{\lambda_\mathrm{re}=0}dP_{\unc,\lambda}(\unc,\lambda) = \\
&\int_{\rho=1} \left(0.5\delta_{(0)}(\unc)+0.5\delta_{(1)}(\unc)\right)d\unc \int_{\lambda_\mathrm{re}=0}\delta_{(0)}(\lambda_\mathrm{re}) d\lambda_\mathrm{re}=0.5.
\end{align*}
Therefore, by exploiting the information on the mean we can reduce the upper probability of instability from $1$ to $0.5$.
 \hfill $\blacksquare$
\end{example} 
}

\section{Solving moment problems through SDP relaxations} \label{Sec:LMIrelaxation}
Note that, in  problem \eqref{eqn:probProbloc}: (i) the decision 
variables are the amount of non-negative mass ${\Pr}_{\unc,x,\lambda}$ assigned to
each point $(\unc,x,\lambda)$, (ii) the objective function and the constraints
 are linear in the optimization variables $P_{\unc,x,\lambda}$.
Therefore, \eqref{eqn:probProbloc} is a \emph{semi-infinite linear
program}, with a finite number of constraints but with  infinite number of decision variables. In this section, we show how to use   results from the theory-of-moments  relaxation proposed by Lasserre in \cite{2001Ala}, and concerning the characterization
of those sequences that are sequence of moments of some probability measures,   to relax the semi-infinite linear programming  problem \eqref{eqn:probProbloc} into a hierarchy  of \emph{semidefinite programming} (SDP) problems of finite dimension. 

Let us first introduce the augmented variable vector $\Zetavec~=~\left[x_{\textrm{re}}^\top \ \ x_{\textrm{im}}^\top  \ \ \unc^\top \ \ \lambda_{\textrm{re}} \ \ \lambda_{\textrm{im}}\right]^\top \in \mathbb{R}^{\sizevec{\Zetavec}}$ (with $\sizevec{\Zetavec}~=~2\sizevec{a}+{\sizevec{\unc}}+2$)  and,  with some abuse of notation, let us define $h(\Zetavec)=\|x\|^2$ and $\tilde{f}(\Zetavec)=f(\unc)$.   Problem   \eqref{eqn:probProbloc}  can be then rewritten in terms of the augmented variable $\Zetavec$ and the cumulative  distribution function $P_{\Zetavec}$ as

\begin{subequations} \label{eqn:probProbsingZetavec}
\begin{align}  
\overline{p}=& \sup_{P_{\Zetavec}} \int h(\Zetavec) dP_\Zetavec(\Zetavec)\\
  \nonumber
  s.t. \\
	      &  \int dP_{\Zetavec}(\Zetavec)=1, \\
  &  \int \tilde{f}_i(\Zetavec)dP_{\Zetavec}(\Zetavec)=\mu_i, \ i=2,\ldots,\sizevec{f},\\
	      &  \int_{\Zvecsupp}dP_{\Zetavec}(\Zetavec)=1, 
\end{align} 
\end{subequations}
 where $\Zvecsupp$ defines the support of the probability measure $\Pr_{\Zvec}$. Thus, based on the definition of the sets $\Unc$ (eq. \eqref{eqn:DeltaS}) and $\mathcal{D}^c$ (eq. \eqref{eqn:diskcompl}), 
the set $\Zvecsupp$ is described by:
\begin{align} \label{eqn:Zdesc}
\Zvecsupp=& \left\{\Zetavec=\left[x_{\textrm{re}}^\top \ \ x_{\textrm{im}}^\top \ \ \unc^\top \ \lambda_{\textrm{re}} \ \ \lambda_{\textrm{im}} \right]^\top: \right. \nonumber \\
      & \ g_i({\unc}) \geq 0, \  \ i=1,\ldots,\sizevec{g}, \nonumber  \\
			& \ d_i(\lambda_{\textrm{re}},\lambda_{\textrm{im}}) \geq 0, \  \ i=1,\ldots,\sizevec{d}, \nonumber  \\
		  &   \left(A(\unc)- \lambda_{\textrm{re}} I\right) x_{\textrm{re}}+ \lambda_{\textrm{im}} x_{\textrm{im}}=0,   \nonumber  \\
			& \ \left(A(\unc)- \lambda_{\textrm{re}} I\right) x_{\textrm{im}}- \lambda_{\textrm{im}} x_{\textrm{re}}=0,  \nonumber	\\
&    \ \left\|x_{\textrm{re}}\right\|^2+\left\|x_{\textrm{im}}\right\|^2 \leq 1 \left. \right\}.
\end{align}
In order to compact the notation, we will rewrite the set $\Zvecsupp$  as:
\begin{align} \label{eqn:Zdescv2}
\Zvecsupp=\left\{\Zetavec \in \mathbb{R}^{\sizevec{z}}: \right. & \ q_j(\Zetavec) \geq 0, \ \ j=1,\ldots,\sizevec{q}  \left. \right\}, 
\end{align}
with $q_j(\Zetavec)$ being real-valued polynomial functions in $\Zetavec$, properly defined  based on the description of $\Zvecsupp$ in \eqref{eqn:Zdesc}.

\begin{example}
Since in the explanatory example considered so far  $A(\unc)$ is a real symmetric matrix, its eigenvalues are real, and thus we considered an augmented variable vector $\Zetavec$:
\begin{equation}
\Zetavec=[\unc \  \lambda_{\textrm{re}} \ x_1 \ x_2]^\top \in \mathbb{R}^4. 
\end{equation}
The objective function $h(z)$ is $h(z)=z_3^2+z_4^2$, and the components of the vector-valued function $\tilde{f}(\Zetavec)$ defining the constraints on the moments is $\tilde{f}_1(z)=1$ and $\tilde{f}_2(z)=z_1$. According to the description in \eqref{eqn:Zdescv2}, the set $\Zvecsupp$ defining the support of the probability measure $\Pr_{\Zvec}$ is given by:
\begin{align} \label{eqn:Zdescex}
\Zvecsupp=& \left\{\Zetavec= [\unc \  \lambda_{\textrm{re}} \ x_1 \ x_2]^\top: \right. \nonumber \\
      & q_1(\Zetavec) \doteq z_1 \geq 0, \ \ q_2(\Zetavec) \doteq 1-z_1 \geq 0,  \nonumber \\
			& q_3(\Zetavec) \doteq  z_2 \geq 0, \nonumber  \\
		  &  q_4(\Zetavec) \doteq  (z_1-1)z_3 \geq 0, \   q_5(\Zetavec) \doteq  -(z_1-1)z_3 \geq 0, \nonumber  \\
		  &  q_6(\Zetavec) \doteq  z_4 \geq 0, \   q_7(\Zetavec) \doteq  -z_4 \geq 0, \nonumber  \\
&    q_8(\Zetavec) \doteq 1-z_3^2-z_4^2 \geq 0 \left. \right\}. \nonumber 
\end{align} \hfill $\blacksquare$\\
\end{example}

For an integer  $\displaystyle \tau \in \mathbb{N}:\  \tau \geq \tilde{\tau}$,  with 
\begin{equation} \label{eqn:deltaunder}
\tilde{\tau}= \max \left\{ 1,\max_{i=1,\ldots,\sizevec{\tilde{f}}}\left\lceil \frac{deg(\tilde{f}_i)}{2}  \right\rceil,  \max_{j=1,\ldots,\sizevec{q}}\left\lceil \frac{deg(q_j)}{2}  \right\rceil  \right\},
\end{equation}
let us rewrite  $h(\Zetavec) \in \setpoly{2\tau}{}[\Zetavec]$ and each component $\tilde{f}_i(\Zetavec) \in \setpoly{2\tau}{}[\Zetavec] $ of the vector-valued function $\tilde{f}(\Zetavec)$ as 

\begin{equation} \label{eqn:hf} h(\Zetavec)=\sum_{\alpha \in \setinteger{2 \tau}{\sizevec{\Zetavec}}}\bm{h}_{\alpha}\Zetavec^{\alpha}, \ \ \ \ \tilde{f}_i(\Zetavec)=\sum_{\alpha \in \setinteger{2 \tau}{\sizevec{\Zetavec}}}\bm{\tilde{f}}_{i,\alpha}\Zetavec^{\alpha}, \end{equation}
\DP{where, according to the notation introduced in Section \ref{Sec:notation},  $\bm{h}_{\alpha}$ (resp. $\bm{\tilde{f}}_{i,\alpha}$) are the  coefficients of the polynomial $h(\Zetavec)$ (resp. $\tilde{f}_i(\Zetavec)$). Based on eq. \eqref{eqn:hf}, we can write
\begin{align*}  
\int h(\Zetavec) dP_{\Zetavec}(\Zetavec)&=\int \left(\sum_{\alpha \in \setinteger{2 \tau}{\sizevec{\Zetavec}}}\!\! \bm{h}_{\alpha}\Zetavec^{\alpha}\right) dP_{\Zetavec}(\Zetavec)=\sum_{\alpha \in \setinteger{2 \tau}{\sizevec{z}}}\!\! \bm{h}_{\alpha}m_{\alpha},
\end{align*}
where $m_{\alpha}$ are the moments of the probability measure  $\Pr_{\Zetavec}$, i.e.,
\begin{equation*}  
m_{\alpha}=\int z^{\alpha} dP_{\Zetavec}(\Zetavec),
\end{equation*}
as introduced in Section \ref{Sec:notation}. Similar considerations hold  for the polynomial $\bm{\tilde{f}}_{i}(\Zetavec)$.

 Thus, solving problem \eqref{eqn:probProbsingZetavec} is equivalent to solve:
\begin{align}  \label{eqn:probProbsingZvecmom}
\overline{p}=& \sup_{m=\{m_{\alpha}\}_{\alpha \in \setinteger{2 \tau}{\sizevec{z}}}}\sum_{\alpha \in \setinteger{2 \tau}{\sizevec{z}}}\bm{h}_{\alpha}m_{\alpha}\\
  \nonumber
  s.t. \\
  &  \sum_{\alpha \in \setinteger{2 \tau}{\sizevec{z}}}\bm{\tilde{f}}_{i,\alpha}m_{\alpha}=\mu_i, \ \ i=2,\ldots,\sizevec{f}, \nonumber\\
				& m \textrm{\ is a sequence of moments generated by a }	\nonumber \\	
				& \ \ \ \  \textrm{probability measure with  support on } \Zvecsupp.  \nonumber
\end{align} 
Comparing \eqref{eqn:probProbsingZetavec} and \eqref{eqn:probProbsingZvecmom} is evident that now the optimization variables are the moments $m_{\alpha}$ (real numbers), where   the constraint ``$\Pr_{\Zetavec}$ is a probability measure on $\Zvecsupp$'' has been replaced by ``$m$ is a sequence
 of moments generated by a probability measure with  support on $\Zvecsupp$''.}

From a straightforward application of Lasserre's hierarchy (see \cite{2001Ala} and \cite[Sec. 4.1.5]{2009Ala}),  necessary conditions for the sequence   $\displaystyle m=\{m_{\alpha}\}_{\alpha \in \setinteger{2 \tau}{\sizevec{z}}}$  to be a sequence of moments generated by a probability measure ${\Pr}_{\Zvec}(\Zvec)$ with support on $\Zvecsupp$ can be derived. Before discussing the application  Lasserre's hierarchy to  problem \eqref{eqn:probProbsingZetavec}, let us introduce the following notation. 

For   a generic polynomial function $g \in \setpoly{\tau}{}[\Zvec]$,  let us define the map $L_{m}(g)$ as:
\begin{align}
g \mapsto L_{m}(g)\!=\!\!\!  \int \! \! g(z)dP_\Zvec(\Zvec)\!=\! \!\!\! \!\!  \sum_{\alpha \in \setinteger{\tau}{\sizevec{z}}} \!\!\! \bm{g}_{\alpha} \!\! \int \! \! \Zvec^{\alpha}dP_\Zvec(\Zvec)\! =\!\!\!\!\!  \sum_{\alpha \in \setinteger{ \tau}{\sizevec{z}}}\!\!\!\! \bm{g}_{\alpha}m_{\alpha}. \nonumber 
\end{align}
 

Let us define the so-called \emph{moment matrix} $M_{\tau}(m)$ truncated to order $\tau$ as
\begin{align} \label{eqn:moment}
M_\tau(m)=\int \canbas{\tau}{}(\Zvec)\canbas{\tau}{\top}(\Zvec) dP_\Zvec(\Zvec)=  L_{m}(\canbas{\tau}{}(\Zvec)\canbas{\tau}{\top}(\Zvec)),
\end{align}
\DP{with $\canbas{\tau}{}(z)$ defined in  Section \ref{Sec:notation} and where the operator $L_{m}$ is applied entry-wise to the matrix $\canbas{\tau}{}(\Zvec)\canbas{\tau}{\top}(\Zvec)$.}

Let us also define the  so-called truncated \emph{localizing matrix} $M_{\tau}(gm)$ of order $\tau$ associated with the polynomial $g$ as:
\begin{align} \label{eqn:locmatrix}
M_{\tau}(gm)=\int \!\! g(\Zvec)\canbas{\tau}{}(\Zvec)\canbas{\tau}{\top}(\Zvec) dP_\Zvec(\Zvec)=L_{m}(g(\Zvec)\canbas{\tau}{}(\Zvec)\canbas{\tau}{\top}(\Zvec)).
\end{align}

Based on the definition of the moment and localizing matrices, the following theorem, which is the basis for the Lasserre's hierarchy \cite{2001Ala}, can be  stated.    

\begin{theorem}  \label{Prop:Momentmatrix} {\normalfont \DP{\textbf{\cite[Sec. 4.1.5]{2009Ala}}}}  If $\displaystyle m=\{m_{\alpha}\}_{\alpha \in \setinteger{2 \tau}{\sizevec{z}}}$ is a sequence of moments generated  by a probability measure ${\Pr}_{\Zvec}(\Zvec)$ supported by  $\Zvecsupp$, then 
\begin{align} \label{eqn:Mrel}
M_\tau(m) \! \succeq \! 0, \  m_{\small 0\cdots 0} \! = \! 1,  \  M_{\tau-\left\lceil \frac{deg(q_j)}{2}  \right\rceil}(q_j m)  \! \succeq \! 0,   j=1,\ldots \sizevec{q},
\end{align}
for any integer $\tau \geq \tilde{\tau}$, with $\tilde{\tau}$ defined in \eqref{eqn:deltaunder}.  
\hfill $\blacksquare$
\end{theorem}
\begin{proof}
  First, observe that if $\displaystyle m=\{m_{\alpha}\}_{\alpha \in \setinteger{2 \tau}{\sizevec{z}}}$ is a sequence of moments generated  by a probability measure ${\Pr}_{\Zvec}$ supported by  $\Zvecsupp$, then:
\begin{align*}
 m_{\alpha}=\int \Zvec^{\alpha} dP_{\Zvec}(\Zvec), \ \ m_{\small 0\cdots 0}=\int_{\Zvecsupp}  dP_{\Zvec}(\Zvec)=1.
\end{align*}
Based on the definition of the moment matrix  $M_{\tau}(m)$ (see \eqref{eqn:moment}),   for any real vector $\bm g$ of proper dimension, we have
\begin{align} \label{eqn:Mpos}
\bm{g}^\top M_{\tau}(m)\bm{g} =&  \int  \bm{g}^\top \canbas{\tau}{}(\Zvec)\canbas{\tau}{\top}(\Zvec)  \bm{g} dP_\Zvec(\Zvec)  \nonumber \\
= & \int g^2(z) dP_\Zvec(\Zvec) \geq 0,
\end{align}
 where $g(z)$ is a generic polynomial in $\setpoly{\tau}{}$, whose vector of coefficients in the canonical basis $\canbas{\tau}{}(\Zvec)$ is $\bm{g}$.   Since condition \eqref{eqn:Mpos} holds for any vector $\bm{g}$, $M_{\tau}(m) \succeq 0$.

For any $j=1,\ldots,\sizevec{q}$, let us now take another real-valued vector $\bm{g}$ of proper dimension, and consider the term
\begin{align} \label{eqn:proofloc}  
\bm{g}^\top M_{\tau-\left\lceil \frac{deg(q_j)}{2}  \right\rceil}(q_j m) \bm{g}.
\end{align}
Based on the definition of the localizing matrix $\displaystyle M_{\tau-\left\lceil \frac{deg(q_j)}{2}  \right\rceil}(q_j m)$ (see eq. \eqref{eqn:locmatrix}), the term  \eqref{eqn:proofloc}  becomes:
\begin{align}
& \bm{g}^\top M_{\tau-\left\lceil \frac{deg(q_j)}{2}  \right\rceil}(q_j m) \bm{g}= \int  q_j(\Zvec)\bm{g}^\top \canbas{\tau}{}(\Zvec)\canbas{\tau}{\top}(\Zvec)  \bm{g} dP_\Zvec(\Zvec)  \nonumber \\
= & \int  q_j(\Zvec)g^2(\Zvec) dP_\Zvec(\Zvec)=\int_{\Zvecsupp}  q_j(\Zvec)g^2(\Zvec) dP_\Zvec(\Zvec) \geq 0,
\end{align}
where the above inequality holds since, by definition of the set $\Zvecsupp$ (eq. \eqref{eqn:Zdescv2}), $q_j(\Zvec) \geq 0$ for any $\Zvec \in \Zvecsupp$.   Thus, $M_{\tau-\left\lceil \frac{deg(q_j)}{2}  \right\rceil}(q_j m)  \succeq 0$.  
\end{proof}
\ \\
 
Based on Theorem \ref{Prop:Momentmatrix}, for any integer $\displaystyle \tau \geq \tilde{\tau}$, 
 instead of requiring the conditions in \eqref{eqn:probProbsingZvecmom}, one may require the weaker conditions in \eqref{eqn:Mrel}. This leads  to an upper bound $\bar{p}^{\tau}$ of $\bar{p}$, which can be computed by solving the (convex) SDP problem:
\begin{subequations}  \label{eqn:probProbsingZvecmomrelax}
\begin{align} 
\bar{p}^{\tau}=& \sup_{m=\{m_{\alpha}\}_{\alpha \in \setinteger{2 \tau}{\sizevec{z}}}}\sum_{\alpha \in \setinteger{2 \tau}{\sizevec{z}}}\bm{h}_{\alpha}m_{\alpha}\\
  \nonumber
  s.t. \\
  &  \sum_{\alpha \in \setinteger{2 \tau}{\sizevec{z}}}\bm{\tilde{f}}_{i,\alpha}m_{\alpha}=\mu_i, \ \ \ i=2,\ldots,\sizevec{f},  \\
	      &   m_{\small 0\cdots 0} =1, \ \   M_\tau(m) \succeq 0, \\
				&   M_{\tau-\left\lceil \frac{deg(q_j)}{2}  \right\rceil}(q_j m)   \succeq 0, \ \ j=1,\ldots \sizevec{q}. 
\end{align} 
\end{subequations}
{\begin{example}
In the explanatory example considered so far,
\begin{align*}
h(z)=& x_1^2+x_2^2=z_3^2+z_4^2=z^{0020}+z^{0002}, \\ 
\tilde{f}_2(z)=& \rho = z_1 = z^{1000}. 
\end{align*}
 Thus, for a relaxation order $\tau=2$, the   SDP problem \eqref{eqn:probProbsingZvecmomrelax}    is given by:
 \begin{align}  \label{prob:exSDP}
\overline{p}^{\tau}=& \sup_{m=\{m_{\alpha}\}_{\alpha \in \setinteger{2 \tau}{4}}} m_{0020}+m_{0002} \nonumber\\
  s.t. \   &  m_{0000}=1, \ \ m_{1000}=0.5 \nonumber\\
				& M_1(m) \succeq 0, \ \   M_{0}(q_j m)   \succeq 0, \ \ j=1,\ldots 7,  
\end{align} 
with 
\begin{align*}
M_1(m)=\left[
\begin{array}{ccccc}
m_{0000} & m_{1000} & m_{0100} & m_{0010} & m_{0001}\\ 
m_{1000} & m_{2000} & m_{1100} & m_{1010} & m_{1001}\\
m_{0100} & m_{1100} & m_{0200} & m_{0110} & m_{0101}\\ 
m_{0010} & m_{1010} & m_{0110} & m_{0020} & m_{0011}\\
m_{0001} & m_{1001} & m_{0101} & m_{0011} & m_{0002}\\  
\end{array}
\right],
\end{align*}
{\small \begin{align*}
& M_0(q_1m)=m_{1000}, \ \ M_0(q_2m)\!=\!1\!-\!m_{1000}, \ \ M_0(q_3m)\!=\!m_{0100}, \\
& M_0(q_4m)=m_{1010}-m_{0010}, \ \  M_0(q_5m)=-m_{1010}+m_{0010}, \\
& M_0(q_6m)=m_{0001}, \ \  M_0(q_7m)=-m_{0001}, \\
& M_0(q_8m)=1-m_{0020}-m_{0002}.
\end{align*}} \hfill $\blacksquare$
\end{example}}

By construction, the moment and the localizing matrices  are such that:
\begin{align*}
M_{\tau+1}(m) \succeq 0 & \Rightarrow M_{\tau}(m) \succeq 0, \\
 M_{\tau+1\!-\!\left\lceil \frac{deg(q_j)}{2}  \right\rceil}\! (q_j m)   \! \succeq \! 0 \! & \Rightarrow \!  M_{\tau-\!\left\lceil \frac{deg(q_j)}{2}  \right\rceil}\! (q_j m)   \! \succeq \! 0. 
\end{align*}

This implies: 
\begin{equation}  \label{eqnp4}
\bar{p}^{\tau} \geq \bar{p}^{\tau+1} \geq \bar{p},
\end{equation}
which means that, as the relaxation order $\tau$ increases, the SDP relaxation \eqref{eqn:probProbsingZvecmomrelax}   becomes tighter.    Furthermore, under mild restrictive assumptions on the description of  the set $\Zvecsupp$, the solution of the SDP relaxed problem  \eqref{eqn:probProbsingZvecmomrelax} converges to the global optimum  $\bar{p}$ of the original optimization problem \eqref{eqn:probProbsingZetavec}, i.e., 
\begin{equation} \label{eqn:con} \lim_{\tau \rightarrow \infty}\bar{p}^{\tau}=\bar{p}. \end{equation}
The proof of the converge property in \eqref{eqn:con} is reported in the appendix, along with the needed assumptions.

\begin{remark}
The number $N_{\tau}$ of the optimization variables $\displaystyle m=\{m_{\alpha}\}_{\alpha \in \setinteger{2 \tau}{\sizevec{z}}}$ of problem \eqref{eqn:probProbsingZvecmomrelax} is given by the binomial expression:
\begin{equation*}
N_{\tau}=\left(\begin{array}{c} \sizevec{z} + 2\tau \\ 2\tau \end{array}\right)=O\left( \sizevec{z}^{2\tau}\right), 
\end{equation*}
and thus, for fixed relaxation order $\tau$, $N_{\tau}$ grows polynomially with the size of the vector $z$.
\end{remark}


\begin{property} \label{remark:conscondppp}
Since the relaxed SDP problem \eqref{eqn:probProbsingZvecmomrelax} provides an upper bound of $\bar{p}$ (i.e., $\bar{p}^{\tau}   \geq \bar{p}$), sufficient conditions on the $\mathcal{D}$-stability of  $A(\unc)$ can be derived from $\bar{p}^{\tau}$. Specifically:
\begin{itemize} 
\item if the only information on the uncertain parameter vector $\unc$ is the support $\Unc$ of its probability measures (i.e., $\unc \in \Unc$), then, from Theorem \ref{theor:loctheorem} and Corollary \ref{theor:loctheorem2},  $\bar{p}$ can be either $0$ ($A(\unc)$ is robustly $\mathcal{D}$-stable) or $1$ ($A(\unc)$ is not robustly $\mathcal{D}$-stable). Thus, if $\bar{p}^{\tau}<1$, we can claim that   $\bar{p}=0$ and thus $A(\unc)$ is guaranteed to be robustly $\mathcal{D}$-stable  against the uncertainty set $\Unc$. On the other hand, if $\bar{p}^{\tau} \geq 1$, no conclusions can be drawn, in principle,  on the robust $\mathcal{D}$-stability of $A(\unc)$. 
\item if the information on the moments of $\unc$ are given, then $\bar{p}$ represents the probability of the matrix $A(\unc)$ to have at least an eigenvalue in $\mathcal{D}^c$.  Thus, since $\bar{p}^{\tau}\geq\bar{p}$, we can claim that $A(\unc)$ is not $\mathcal{D}$-stable with probability less than or equal to $\bar{p}^{\tau}$. Equivalently,  $A(\unc)$ is $\mathcal{D}$-stable with probability  at least $1-\bar{p}^{\tau}$. \hfill $\blacksquare$
\end{itemize}
\end{property}

{\begin{example}
Let us go back to the explanatory example. For a relaxation order $\tau=2$, the solution of the SDP problem \eqref{prob:exSDP} is $\overline{p}^{\tau}=0.5$. Thus,  we can claim that $A(\unc)$ is not $\mathcal{D}$-stable with probability  at most $=0.5$. Note that the obtained solution $\overline{p}^{\tau}=0.5$ is tight (i.e., $\overline{p}^{\tau}=\overline{p}$). In fact, we have already seen in Example \ref{ex:Probc2} that, for  a probability measure
${\Pr}_{\unc}=0.5\delta_{(0)}+0.5\delta_{(1)}$, the matrix $A(\unc)$ has an eigenvalue equal to $0$ with probability $0.5$.  \hfill $\blacksquare$
\end{example}}

\section{Applications and examples} \label{Sec:examples}
 

In this section, we show the application of the proposed approach \DPrev{through three numerical examples. The problem of robust Hurwitz stability analysis of uncertain matrices is addressed in the first example, and a comparison with the polynomial optimization based approaches proposed in \cite{hess2016semidefinite,henrion2012inner} is also provided.   Robust and probabilistic analysis of the properties of dynamical models with parametric uncertainty is discussed in the second and in the third  examples. Specifically, in the second example, taken from   \cite{KiBr2016}, sufficient
conditions for nonexistence of bifurcations in uncertain nonlinear
continuous-time dynamical systems are derived. Both  the deterministic  and the probabilistic scenario are considered. The other example is focused on the analysis of robust stability and   performance verification of  LTI systems with parametric uncertainty. The robust and  probabilistic  formulations are combined to verify robust stability  of the system and to compute the minimum probability to meet the   performance specifications.} 

All computations are carried out on  an i7 2.40-GHz Intel core
processor with $3$ GB of RAM running MATLAB R2014b. The YALMIP  Matlab interface \cite{lofberg2004yalmip} is used to construct the relaxed SDP 
problems \eqref{eqn:probProbsingZvecmomrelax}, which are solved  through the general purpose SDP
solver SeDuMi \cite{1999Ast}.


\DPrev{\subsection{Hurwitz stability and polynomial abscissa} \label{subsec:Henrion}
The aim of this  example is to highlight the  advantages of our approach    w.r.t. the polynomial optimization based methods presented in \cite{hess2016semidefinite,henrion2012inner}.  Since the method in \cite{hess2016semidefinite} is focused on the approximation of the abscissa of an uncertain  polynomial (i.e., maximum real part of the roots of a univariate polynomial), a robust Hurwitz stability analysis problem is discussed.
 
Let us consider the uncertain matrix 
\begin{equation} \label{eqn:Arho}
A(\rho)=\left[\begin{array}{cc}
-2.4-\rho_1^2  & 6-\rho_1^2 \\
1-2\rho_1^2  & -2.9-2\rho_1
\end{array}
\right],
\end{equation} 
with $\rho_1 \in \Delta=\left[-0.1 \ \ 3.4 \right]$,  whose characteristic polynomial is given by:
\begin{align} \label{eqn:expPOL}
P(s,\rho_1)=&s^2+(5.3+2\rho_1+\rho_1^2) s\\
           + & 0.96+4.8\rho_1+15.9\rho_1^2+2\rho_1^3-2\rho_1^4.  \nonumber
\end{align}

\noindent \underline{\emph{Polynomial abscissa approximation \cite{hess2016semidefinite}}} \\
\noindent The main idea in \cite{hess2016semidefinite} is to find a fixed-degree polynomial $\overline{P}_d(\rho_1)$ approximating, from above, the abscissa   $a(\rho_1)$  of the polynomial $P(s,\rho_1)$. Specifically, among all the polynomials $\overline{P}_d(\rho_1)$ of given degree $d$ such that
\begin{equation*}
\overline{P}_d(\rho_1) \geq a(\rho_1) \ \ \forall \rho_1 \in \Delta,
 \end{equation*}
the one minimizing the integral 
\begin{equation} \label{eqn:exint1}
\int_{\rho_1 \in \Delta} \overline{P}_d(\rho_1) d\rho_1 
 \end{equation}
is sought. SDP relaxations based on \emph{sum-of-squares} are then used to find the upper approximating polynomial $\overline{P}_d(\rho_1)$. Note that, if $  \max_{\rho_1 \in \Delta} \overline{P}_d(\rho_1)<0$, then all the roots of the characteristic polynomial $P(s,\rho_1)$ have negative real part, thus the matrix $A(\rho)$ in \eqref{eqn:Arho} is guaranteed to be robust Hurwitz stable. Fig. \ref{fig:abscissa} shows the abscissa $a(\rho_1)$ of the polynomial $P(s,\rho_1)$, along with computed upper approximating polynomial $\overline{P}_d(\rho_1)$ of degree $d=8$ in the interval $\Delta=\left[-0.1 \ \ 3.4 \right]$. 
 As $\overline{P}_d(\rho_1) \geq 0$ for some values of $\rho_1 \in \Delta$, no conclusions can be drawn from  $\overline{P}_d(\rho_1)$ on robust Hurwitz stability of the matrix $A(\rho)$. This conservativeness is due to the fact that  the computed   polynomial $\overline{P}_d(\rho_1)$ is the ``best''  (w.r.t. the integral \eqref{eqn:exint1}) upper approximation of the abscissa $a(\rho_1)$   over the whole uncertainty set $\Delta$. On the other hand, in assessing  robustly Hurwitz stability of $P(s,\rho_1)$, we are only interested in approximating   the   maximum of the abscissa over $\rho_1 \in \Delta$. The CPU time required to verify  Hurwitz stability of $A(\rho)$ is $2.5$ s. This includes the time required to compute the upper approximating polynomial $\overline{P}_d(\rho_1)$ as well as the time required to compute  its maximum over $\rho_1$  through Lasserre's relaxation. 

Note that, in the general case of multidimensional
uncertainty $\rho$, another source of conservativeness may also come from the fact that the maximum of the polynomial $\overline{P}_d(\rho)$ over $\Delta$ cannot be computed with a simple plot, but it should be computed through the Lasserre's SDP relaxation \cite{2001Ala}, which only provides an upper bound of the maximum of  $\overline{P}_d(\rho)$. Finally, in case the polynomial $\overline{P}_d(\rho)$ is of large degree (say, $d>10$), a large Lasserre's relaxation order may be needed to achieve a tight approximation of the maximum of $\overline{P}_d(\rho)$, thus leading to  Lasserre's relaxations which might be computationally intractable.\\  

\begin{figure}[htp]
\centering
\centering
 \includegraphics[ trim={0cm 0cm 9cm 25cm},clip]{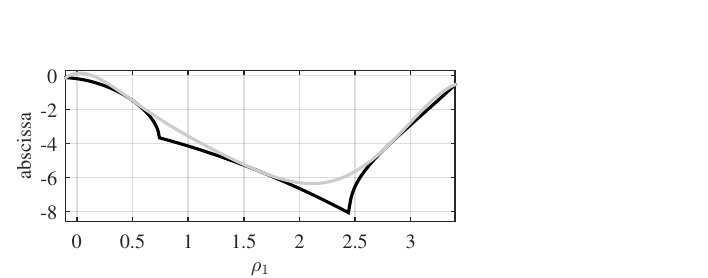}
 \caption{Abscissa of the polynomial $P(s,\rho_1)$ in \eqref{eqn:expPOL} (black line) and computed polynomial approximation (gray line).} \label{fig:abscissa}
\end{figure} 

\noindent \underline{\emph{Hermite stability criterion \cite{henrion2012inner}}} \\
\noindent The problem of robust $\mathcal{D}$-stability of a polynomial is tackled  in \cite{henrion2012inner} approximating the minimum  eigenvalue of the associated  Hermite matrix. In order to check Hurwitz stability of the 
polynomial $P(s,\rho_1)$ in \eqref{eqn:expPOL},  the associated $2\times 2$ symmetric Hermite matrix $H(\rho_1)$ is constructed. Since the coefficients of $P(s,\rho_1)$ are polynomials in $\rho_1$ of maximum degree $4$, the entries of the matrix $H(\rho_1)$ are polynomials in $\rho_1$ of maximum degree $8$. According to the Hermite stability criterion (see \cite{henrion2012inner}),  $P(s,\rho_1)$ is robustly Hurwitz stable if and only if
\begin{equation} \label{rqn:exH}
H(\rho_1) \succ 0, \ \ \forall \rho_1 \in \Delta.
\end{equation}
  The  robust minimum eigenvalue of $H(\rho_1)$ is given by
\begin{align} \label{eqn:lambdapol}
\lambda_{\mathrm{min}}=\min_{\rho \in \Delta}\min_{x \in \mathbb{R}^2: x^\top x=1} x^\top H(\rho_1) x.
\end{align}
As well known, \eqref{rqn:exH} holds, or equivalently $P(s,\rho_1)$ is robustly Hurwitz stable, if and only if  $\lambda_{\mathrm{min}}>0$. 
Then, a lower bound $\underline{\lambda}_{\mathrm{min}}$ of $\lambda_{\mathrm{min}}$ is computed solving the polynomial optimization problem \eqref{eqn:lambdapol} through the Lasserre's hierarchy, for a relaxation order $\tau=5$, which is the minimum  allowed value for $\tau$, as the objective function in \eqref{eqn:lambdapol} is a 10-degree polynomial in the augmented variable  $[x\  \rho_1]$.  We obtain a lower bound $\underline{\lambda}_{\mathrm{min}}=6.5$, in a CPU time of $2.7$ s. The obtained results allow us to claim that  $H(\rho_1)$ is robustly positive definite, thus $P(s,\rho_1)$ is robustly Hurwitz stable, and and  no conservativeness is introduced  in relaxing problem \eqref{eqn:lambdapol} through the Lasserre's hierarchy. However, the example   shows that even if the entries of the matrix $A(\rho)$  are polynomial functions of $\rho_1$ of degree at most $2$, the objective function minimized in \eqref{eqn:lambdapol} is a polynomial of degree $10$, which required to use a   Lasserre's relaxation order at least equal to $\tau=5$. As already discussed, the Lasserre's hierarchy may become  computationally intractable in the more general case of multidimensional uncertain parameter $\rho$ and large relaxation orders.\\
 

\noindent \underline{\emph{Robust $\mathcal{D}$-stability analysis}} \\
\noindent The approach proposed in this paper is now used to assess robust Hurwitz stability of the matrix $A(\rho)$.  The polynomial optimization problem \eqref{eqn1:theloc} is formulated, and solved through the Lasserre's hierarchy for a relaxation order $\tau=3$ (namely, the SDP problem \eqref{eqn:probProbsingZvecmomrelax}  is solved without using any information on the moments of $\rho_1$). The obtained solution of the SDP relaxed problem \eqref{eqn:probProbsingZvecmomrelax} is $10^{-9}$. Thus, according to Property \ref{remark:conscondppp},  $A(\rho)$ is robustly Hurwitz stable. The CPU time required to assess robust Hurwitz stability of   $A(\rho)$ is $1.5$ s. Thus, in this simple example, the proposed  approach is about $1.6$x faster than the methods   \cite{henrion2012inner} and  \cite{hess2016semidefinite}. This is due to the fact that, in the presented approach, the Lasserre's relaxation order $\tau$ can be kept ``small'', as the maximum degree of the polynomial constraints in  \eqref{eqn1:theloc} is $3$ because of the product  $A(\rho)x$.  
 }
 

\subsection{Bifurcation analysis}
The example discussed in this section has been recently studied in \cite{KiBr2016}, where the analysis of the location of the eigenvalues of an uncertain matrix is applied to derive sufficient conditions for nonexistence of bifurcations in nonlinear continuous-time dynamical systems with parametric uncertainty.

As an example, \cite{KiBr2016} considers  a  continuous-time  predator-prey model, described by the differential equations
\begin{subequations} \label{eq:exPP}
\begin{align}
\dot{r}_1=&\gamma r_1(1-r_1)-\frac{\rho_1r_1r_2}{\rho_2+r_1},\\
\dot{r}_2=&-\rho_3 r_2+\frac{\rho_1r_1r_2}{\rho_2+r_1},
\end{align}
\end{subequations}
where $r_1$ and $r_2$ are scaled population numbers, $\gamma=0.1$ is the prey growth rate,  $\rho_1$, $\rho_2$ and $\rho_3$ are real uncertain parameters.   

A non-trivial equilibrium point for the model \eqref{eq:exPP} is:
\begin{align} \label{eq:exFP}
r_{1,\textrm{eq}}=\frac{\rho_2\rho_3}{\rho_1-\rho_3}, \ \ r_{2,\textrm{eq}}=\frac{\gamma\rho_2}{\rho_1-\rho_3}\left(1-\frac{\rho_2\rho_3}{\rho_1-\rho_3}\right).
\end{align}
The Jacobian $J$ of the system  at the equilibrium point $(r_{1,\textrm{eq}},r_{2,\textrm{eq}})$ in \eqref{eq:exFP} is 
\begin{equation} \label{eqn:ex1J}
J(r_{1,\textrm{eq}},r_{2,\textrm{eq}})=
\left[\begin{array}{cc}
\gamma\frac{\rho_3}{\rho_1}\left(1-\rho_2\frac{\rho_1+\rho_3}{\rho_1-\rho_3}\right) & -\rho_3 \\
\gamma \frac{1}{\rho_1}\left(\rho_1-\rho_3-\rho_2\rho_3\right) & 0
\end{array}\right]. 
\end{equation}

Well known  results from the bifurcation theory \cite{kuznetsov2013elements} state that a sufficient condition  to guarantee the existence of no local bifurcations at the equilibrium point $(r_{1,\textrm{eq}},r_{2,\textrm{eq}})$ is that $J(r_{1,\textrm{eq}},r_{2,\textrm{eq}})$ has no eigenvalues with zero real part.

Let us consider uncertain parameters  $\rho_1$, $\rho_2$ and $\rho_3$ which take values in the intervals
\begin{equation} \label{eq:exuncint}
\rho_i \in \left[\rho^{\mathrm{o}}_i - k \Delta \rho^{\mathrm{}}_i  \ \  \rho^{\mathrm{o}}_i + k \Delta \rho^{\mathrm{}}_i  \right], \ \ i=1,2,3,
\end{equation}
where $\rho^{\mathrm{o}}_i$ denotes the nominal value of the parameter  $\rho_i$, $k \in \mathbb{R}$ is a scaling factor, and  $\Delta \rho^{\mathrm{}}_i$ characterizes the  width of the uncertainty interval where $\rho_i$ belongs to. Like in \cite{KiBr2016}, we assume that  the uncertainty intervals in \eqref{eq:exuncint} share the same width, i.e.,  $\Delta \rho^{\mathrm{}}_i=1$ for all $i=1,2,3$, and they are centered at the nominal values  $\rho^{\mathrm{o}}_1=9$,  $\rho^{\mathrm{o}}_2=2$ and $\rho^{\mathrm{o}}_3=2$. 

Note that the entries of the \DP{Jacobian} $J(r_{1,\textrm{eq}},r_{2,\textrm{eq}})$ are not polynomial functions in the uncertain parameters  $\rho_1$, $\rho_2$ and $\rho_3$. However, by introducing  the slack variables: 
\begin{align*}
t_1=\frac{\rho_3}{\rho_1},   \ \ t_2=\frac{1}{\rho_1-\rho_3},
\end{align*}
 the entries of the matrix $J(r_{1,\textrm{eq}},r_{2,\textrm{eq}})$ can be rewritten as polynomial functions in $\rho_1$, $\rho_2$,  $\rho_3$,  $t_1$, $t_2$, i.e., 
\begin{equation*}
J(r_{1,\textrm{eq}},r_{2,\textrm{eq}})=
\left[\begin{array}{cc}
\gamma t_1\left(1-\rho_2\left(\rho_1+\rho_3\right)t_2\right) & -\rho_3 \\
\gamma \left(1-t_1-\rho_2t_1\right) & 0
\end{array}\right], 
\end{equation*}
where the additional polynomial constraints:
\begin{equation} \label{eqn:newcons}
\rho_1t_1=\rho_3,  \ \ (\rho_1-\rho_3)t_2=1,
\end{equation} 
 have to be \DP{considered along with  the interval constraints \eqref{eq:exuncint} on $\rho_i$  to maintain the relationship among the entries of  the matrix $J(r_{1,\textrm{eq}},r_{2,\textrm{eq}})$. This leads to an augmented set of uncertain variables (namely, $\rho_1,\rho_2,\rho_3,t_1,t_2$), which are constrained  to belong to the nonconvex uncertainty set described by the constraints \eqref{eq:exuncint} and \eqref{eqn:newcons}.
} \\

\noindent \underline{\emph{Deterministic bifurcation analysis}} \\
\noindent Let   $\mathcal{D}^c$ be the imaginary axis of the complex place. i.e.,
\begin{align*}  
\mathcal{D}^c=&\left\{\!\lambda \in \mathbb{C}: \! \lambda=\lambda_{\mathrm{re}}+j\lambda_{\mathrm{im}}, \ \ \lambda_{\mathrm{re}}, \lambda_{\mathrm{im}} \! \in \! \mathbb{R}, \ \lambda_{\mathrm{re}}=0\right\},  
\end{align*} 
For  fixed width $k$ of the uncertainty intervals $\displaystyle \left[\rho^{\mathrm{o}}_i - k \Delta \rho^{\mathrm{}}_i  \ \  \rho^{\mathrm{o}}_i + k \Delta \rho^{\mathrm{}}_i  \right]$, the  deterministic bifurcation analysis problem can be formulated as a $\mathcal{D}$-stability analysis problem, or equivalently, in terms of  problem \eqref{eqn:probProbloc}, by  assuming to know only the support of the uncertain parameters $\rho_1$, $\rho_2$,  $\rho_3$. An upper bound $\overline{p}^{\tau}$ of $\overline{p}$ (i.e., solution  of \eqref{eqn:probProbloc}) is computed by solving  the relaxed SDP problem \eqref{eqn:probProbsingZvecmomrelax} for a relaxation order $\tau=3$.

 Based on considerations given in Property \ref{remark:conscondppp}, if $\overline{p}^{\tau} < 1$, then $J(r_{1,\textrm{eq}},r_{2,\textrm{eq}})$ is guaranteed   to have no eigenvalues  on the imaginary axis for any  $\displaystyle \rho_i \in \left[\rho^{\mathrm{o}}_i - k \Delta \rho^{\mathrm{}}_i  \ \  \rho^{\mathrm{o}}_i + k \Delta \rho^{\mathrm{}}_i  \right]$. A bisection on the width $k$ of the uncertainty intervals $\displaystyle \left[\rho^{\mathrm{o}}_i - k \Delta \rho^{\mathrm{}}_i  \ \  \rho^{\mathrm{o}}_i + k \Delta \rho^{\mathrm{}}_i  \right]$ is then carried out to compute (a lower bound of) the maximum value of $k$ such that $J(r_{1,\textrm{eq}},r_{2,\textrm{eq}})$ is guaranteed not to have any eigenvalues  on the imaginary axis for any $\displaystyle \rho_i \in \left[\rho^{\mathrm{o}}_i - k \Delta \rho^{\mathrm{}}_i  \ \  \rho^{\mathrm{o}}_i + k \Delta \rho^{\mathrm{}}_i  \right]$.  The obtained   value of $k$ is $k=0.4620$ (similar to the result obtained in \cite{KiBr2016}) and  the CPU time required to solve problem  \eqref{eqn:probProbsingZvecmomrelax} for fixed $k$ is, in average, $536$ seconds. 
Since  sufficient conditions on robust $\mathcal{D}$-stability are derived from $\overline{p}_{}^{\tau}$, we can claim that the system is guaranteed to have no local bifurcation at the   equilibrium point  $(r_{1,\textrm{eq}},r_{2,\textrm{eq}})$ for any $\displaystyle \rho^{\mathrm{}}_i$ in the interval $\displaystyle \left[\rho^{\mathrm{o}}_i - k \Delta \rho^{\mathrm{}}_i  \ \  \rho^{\mathrm{o}}_i + k \Delta \rho^{\mathrm{}}_i  \right]$, with $i=1,2,3$ and $k=0.4620$.

In this example,   tightness of the computed solution  can be verified analytically. In fact, the determinant of $J(r_{1,\textrm{eq}},r_{2,\textrm{eq}})$ is:
\begin{equation} \label{eq:detJ}
det(J(r_{1,\textrm{eq}},r_{2,\textrm{eq}}))=\alpha\frac{\rho_3}{\rho_1}\left(\rho_1-\rho_3-\rho_2\rho_3\right),
\end{equation}
which is equal to zero for $\rho_1=8.5412$, $\rho_2=\rho_3=2.4650$. This values of $\unc_1,\unc_2$ and $\unc_3$ lie in the intervals $\displaystyle \left[\rho^{\mathrm{o}}_i - k \Delta \rho^{\mathrm{}}_i  \ \  \rho^{\mathrm{o}}_i + k \Delta \rho^{\mathrm{}}_i  \right]$ for $k=0.4650$.\\

\noindent \underline{\emph{Probabilistic bifurcation analysis}} \\
 Let us now consider the case where the uncertain parameters $\unc_i$ belong to the intervals
 \begin{equation*} 
\rho_i \in \left[\rho^{\mathrm{o}}_i - k \Delta \rho^{\mathrm{}}_i  \ \  \rho^{\mathrm{o}}_i + k \Delta \rho^{\mathrm{}}_i  \right], \ \ i=1,2,3,
\end{equation*}
with $k=1$. The expected values of all the three uncertain parameters are known and equal to their nominal values, i.e.,
\begin{equation*}
\mathbb{E}\left[\rho_i\right] =\rho^{\mathrm{o}}_i,  \ \ i=1,2,3.
\end{equation*} 
 Furthermore, we  assume that an upper bound $\overline{\sigma}^2$ on the variance of  the probability measure ${\Pr}_{\unc}$ describing the  parameters $\unc_i$ is known, i.e.,
 \begin{equation*}
\int_{\Unc}\left(\rho_i-\rho^{\mathrm{o}}_i\right)^2dP_\unc(\unc) \leq \overline{\sigma}^2, \ \ i=1,2,3.
\end{equation*}

The solution $\overline{p}_{}^{\tau}$ of the corresponding SDP problem \eqref{eqn:probProbsingZvecmomrelax} is computed for a relaxation order $\tau=3$ and for different values of the (upper bound on the) variance   $\overline{\sigma}^2$. Fig. \ref{fig:Bif} shows the computed upper probability $\overline{p}_{}^{\tau}$ of the system to have a local bifurcation at the equilibrium point $(r_{1,\textrm{eq}},r_{2,\textrm{eq}})$ for different values of the variance $\overline{\sigma}^2$. It can be observed that, although for a width $k=1$ of the uncertainty intervals the system is not guaranteed to have no local bifurcation, under the considered  assumptions on the mean and the maximum variance,  the probability that the system has a local bifurcation at the equilibrium point $(r_{1,\textrm{eq}},r_{2,\textrm{eq}})$ is, in the worst-case scenario, smaller than $0.1$  for $\overline{\sigma}^2$ smaller than $0.15^2$.
\textit{In other words the system has not local bifurcation with probability at least $0.9$}.  Therefore with such an information on the moments we can guarantee
that the system has not local bifurcations with probability at least $0.9$, considering an interval width $k=1$ (that is more than two times the one considered
in the deterministic case ($k=0.462$)). \textit{We can thus be much less conservative and  at the same guaranteeing no bifurcation with ``high probability''.}
Note also that, for values of  $\overline{\sigma}^2$ larger than $0.67^2$, the (upper)   probability $\overline{p}_{}^{\tau}$ of having a local bifurcation saturates to  $0.68$. This seems to  indicate that, above a threshold $\overline{\sigma}^2=0.67^2$, the probability of having a local bifurcation does not increase as the set of feasible probability measures ${\Pr}_\unc$ enlarges. 

\begin{figure}
\centering
\includegraphics{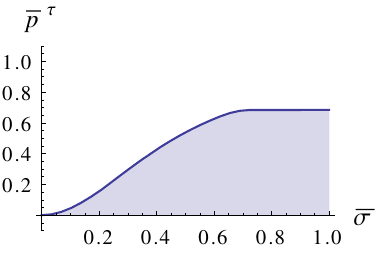}
 \caption{Bifurcation analysis: maximum standard deviation $\overline{\sigma}$ of the probability measures ${\Pr}_\unc$ vs.\ upper probability $\overline{p}_{}^{\tau}$ of having a local bifurcation.} \label{fig:Bif}
\end{figure}


\subsection{Robust stability and performance analysis of LTI systems}

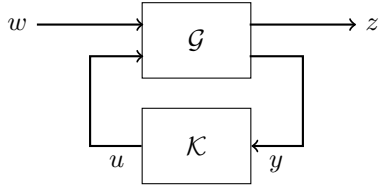
\begin{figure} 
\centering
\begin{tikzpicture}[scale=0.7]
\node[] (G) at  (6,4) {};
\node[] (K) at  (6,2) {};
\node [draw, rectangle,minimum width=1.43cm,minimum height=1cm] (Gp)  at (G) {$\mathcal{G}$};
\node [draw, rectangle,minimum width=1.43cm,minimum height=1cm] (Kp)  at (K) {$\mathcal{K}$};
\draw[->,thick] (Kp)-- (4,2)--(4,3.7)--(5,3.7);
\draw[->,thick] (7,3.7)--(8,3.7)--(8,2)--(Kp);
\draw[->,thick] (3,4.3)--(5,4.3);
\draw[->,thick] (7,4.3)--(9,4.3);
\node[left] at (3,4.3) {$\hspace{2cm} w$};
\node[right] at (9,4.3) {$z$};
\node[below] at (4.5,2) {$u$};
\node[below] at (7.5,2) {$y$};
\end{tikzpicture}
\caption{Feedback control system. $\mathcal{G}$: plant; $\mathcal{K}$: controller; $w$: generalized disturbance; $u$: control input; $z$: controlled output; $y$: measured output.} \label{fig:clsystem}
\end{figure}
 
In this example, we show how the proposed approach can be used to check robust stability and (probabilistic) satisfaction of performance requirements in  uncertain LTI systems.

Consider the closed-loop system depicted in Fig. \ref{fig:clsystem}. The state-space representation of the plant $\mathcal{G}$ is given by:
\begin{subequations} \label{ex2:matrix}
\begin{align}
\dot{x}=& Ax+B_uu+B_ww,\\
\left[\begin{array}{c}
z \\
y
\end{array}\right]=& \left[\begin{array}{c}
C_z \\
C_y
\end{array}\right]x,
\end{align}
\end{subequations}
  \DPrev{where $x=\left[x_1 \ x_2 \ x_3 \ x_4 \right]^\top$ denotes the state of the system, $y=\left[y_1 \ y_2 \ y_3 \ y_4 \right]^\top$ is the measured output that enters the controller $\mathcal{K}$, $u$, $w$ and $z$ are the control input, generalized disturbance and the controlled output, respectively.  
The values of the matrices in \eqref{ex2:matrix} are:
\begin{align*}
A\! \! = \! \! & \left[\!\!\! \begin{array}{cccc}
0 &  1\!+\!0.2\unc_1\!-\!0.1\unc_2  & -0.5 & 3\unc_3\\
\unc_2 & -\!0.2\!+\!0.1\unc_3\!-\!0.3\unc_1 & -0.4 & -\!10 \\
-4  & -\!0.1\!+\!\unc_4\!-\!0.5\unc_2 & -0.5 & 1.5 \\
0.4\!+\!0.2\unc_2\unc_3 & 3 & 4\!+\!0.5\unc_1 & 1\!+\!\unc_4^2
\end{array}\!\!\! \right]  \\
  B_u=&  \left[\!\! \begin{array}{cccc}
1 \;
1 \;
0 \;
1
\end{array}\!\!\right]^T\!,  \  B_w= \left[\! \!\begin{array}{cccc}
1.25 \;
1.25 \;
1.25  \;
1.25 
\end{array}\!\!\right]^T\!, \\
C_z=& \left[ \begin{array}{cccc}
1.25 & 0 & 0 & 0
\end{array}\right], \ \ C_y=\text{diag}(\left[ \begin{array}{cccc}
1 & 1 & 1 & 1
\end{array}\right]). 
\end{align*}
The  parameters $\unc_1$, $\unc_2$, $\unc_3$ and $\unc_4$ defining the dynamic  matrix $A$ are not known exactly 
and they   belong to the uncertainty
intervals 
\begin{subequations} \label{ex2:uncpar}
\begin{align} 
\unc_1 \in \left[\unc_1^{\mathrm{o}}-0.15 \ \ \unc_1^{\mathrm{o}}+0.15 \right],  \ \ & \unc_2 \in \left[\unc_2^{\mathrm{o}}-0.05 \ \ \unc_2^{\mathrm{o}}+0.05 \right],\\
\unc_3 \in \left[\unc_3^{\mathrm{o}}-0.25 \ \ \unc_3^{\mathrm{o}}+0.25 \right],  \ \ & \unc_4 \in \left[\unc_4^{\mathrm{o}}-0.05 \ \ \unc_4^{\mathrm{o}}+0.05 \right],
\end{align} 
\end{subequations}
where $\unc_1^{\mathrm{o}}=1$, $\unc_2^{\mathrm{o}}=0$, $\unc_3^{\mathrm{o}}=0$ and $\unc_4^{\mathrm{o}}=0$  are the nominal values of the parameters. 

The controller $\mathcal{K}$ is  a static output-feedback controller (i.e., $u=-Ky=-Kx$) designed to place the poles of the nominal closed-loop system at $-0.5 \pm j$, $-5$ and $-5$. This is achieved for a matrix gain $K=\left[36.45 \ \   -5.33  \ \ -30.67 \ \  -11.12 \right]$. 

In order to verify the robust stability of the closed-loop system, we check if the  (uncertain) closed-loop dynamic matrix
\begin{equation*}
A_{\textrm{cl}}=A-B_uK
\end{equation*}
has no eigenvalues with positive or null  real part.  This equivalent to verify that the solution $\bar{p}_{}$ of the optimization problem \eqref{eqn:probProbloc} is $0$, where the only information used in \eqref{eqn:probProbloc} is  the uncertainty intervals where the parameters $\unc_1$, $\unc_2$, $\unc_3$ and $\unc_4$ are supposed to belong to, and $\mathcal{D}^c$ is the closed right-half plane of the complex plane. Thus, based on the considerations in Property \ref{remark:conscondppp}, a sufficient condition  to guarantee that $\overline{p}_{}=0$ (or equivalently, the system is robustly stable) is    $\overline{p}_{}^\tau < 1$.  By solving 
problem  \eqref{eqn:probProbsingZvecmomrelax} for a relaxation order $\tau=2$, we obtain $\overline{p}_{}^\tau=0.05$ (CPU time=$44.98$ seconds), thus proving robust stability of the closed-loop system.\\

\noindent \underline{\emph{Robust and probabilistic performance analysis}} \\
Like in  H$_{\infty}$-control design,  the performance of the closed-loop system are specified in terms of the H$_{\infty}$-norm of the closed-loop system $\mathcal{G}_{\textrm{cl}}$ relating the generalized disturbance $w$ and the controlled output $z$, whose state-space representation is given by: 
\begin{align*}
\dot{x}=& (\underset{A_{\textrm{cl}}}{\underbrace{A-B_uK}})x+B_ww,\\
 \begin{array}{c}
z  
\end{array} =&  \begin{array}{c}
C_z  
\end{array} x.
\end{align*} 
For a given $\eta > 1$, we claim that robust performance is  achieved if  
\begin{equation*}
\left\|\mathcal{G}_{\textrm{cl}}\right\|_{\infty} < \eta,   
\end{equation*}
 for all values taken by the   parameters $\unc_1,\unc_2,\unc_3,\unc_4$  in the uncertainty intervals in \eqref{ex2:uncpar}.

As well known in  the H$_{\infty}$-control theory, the condition $\left\|\mathcal{G}_{\textrm{cl}}\right\|_{\infty} < \eta$ holds if and only if  the \emph{Hamiltonian} matrix 
\begin{equation*}
H=\left[\begin{array}{cc}
A_{\textrm{cl}} & \frac{1}{\eta^2}B_wB_w^\top \\
-C_zC_z^\top & -A_{\textrm{cl}}^\top
\end{array}\right]
\end{equation*}
has no eigenvalues on the imaginary axis. Let us set $\eta=1$. By solving the corresponding SDP relaxed problem
  \eqref{eqn:probProbsingZvecmomrelax} for a relaxation order $\tau=2$, we obtain $\bar{p}^\tau=1$. Thus, in principle,  we cannot draw any conclusions on the robust performance of the system. 
	
	 Nevertheless, some heuristics  can be used to verify, from the solution of problem \eqref{eqn:probProbsingZvecmomrelax}, if the Hamiltonian $H$ has some eigenvalues on the imaginary axis. In fact, when Lasserre's hierarchy is used to relax (deterministic) polynomial optimization problems (like \eqref{eqn1:theloc}),   the first order moments $m_{\alpha}$ of the SDP relaxed problem \eqref{eqn:probProbsingZvecmomrelax}  provides, in practice, a good approximation $(\hat{\unc},\hat{x},\hat{\lambda})$ of the global minimizer  $(\unc^*,x^*,\lambda^*)$ of the  original optimization problem \eqref{eqn1:theloc}. 
	By looking at the first order moments  $m_{\alpha}$ associated to the uncertain parameters $\unc$, we obtain   
	\begin{align*}
	\hat{\unc}=\left[  
	  \hat{\unc}_1\;\;
    \hat{\unc}_2 \;\;
   \hat{\unc}_3\;\;
   \hat{\unc}_4
	\right]^T=\left[   
	  1.101\;\;
    0.047 \;\;
   -0.222\;\;
   -0.005
\right]^T.
	\end{align*}   For this values of the uncertainty $\unc$, we obtain $\left\|\mathcal{G}_{\textrm{cl}}\right\|_{\infty}=1.013$. Thus, we can claim that   robust performance requirements are not achieved.

	Finally, the probabilistic framework is considered. Probabilistic conditions on the  performance of the system are derived under the assumption that the expected value of the uncertain parameters is given by their nominal parameters $\unc_1^{\mathrm{o}}=1$, $\unc_2^{\mathrm{o}}=\unc_3^{\mathrm{o}}=\unc_4^{\mathrm{o}}=0$, and the maximum  variance $\overline{\sigma}_i^2$  of  the probability measures ${\Pr}_{\unc_i}$ describing the uncertain  parameters $\unc_1$, $\unc_2$, $\unc_3$ and $\unc_4$  is available. Specifically, 
 \begin{align}
\int_{\Unc}\left(\rho_i-\rho^{\mathrm{o}}_i\right)^2dP_{\unc_i}(\unc_i) \,\leq\, & \overline{\sigma}_i^2, 
\end{align} 
with $\overline{\sigma}_1=0.024,\overline{\sigma}_2=,0.008,\overline{\sigma}_3=0.040,\overline{\sigma}_4=0.008$.

Solving the corresponding SDP  problem \eqref{eqn:probProbsingZvecmomrelax} for a relaxation order $\tau=2$, we obtain  $\overline{p}^\tau=0.082$ (CPU time=4916 seconds).  Based on the obtained results,\textit{ we can claim that the closed-loop system is guaranteed to be robustly stable and the performance requirements are fulfilled with probability at least $0.918$}.  

For a more exhaustive analysis on the performance of the system, we also compute the (minimum) probability $\underline{p}^\tau=1-\overline{p}^\tau$ to satisfy the condition $\left\|\mathcal{G}_{\textrm{cl}}\right\|_{\infty} < \eta$ for different values of $\eta$. The obtained results are reported in Fig. \ref{fig:Control}, which shows the computed $\underline{p}^\tau$ (representing a lower bound on the probability $\textrm{Pr}_{\rho}(\|G_\mathrm{cl}\|_{\infty}< \eta)$) w.r.t. different values  
of the norm bound $\eta$. Note that, for $\eta\geq 1.4$, the constraint $\left\|\mathcal{G}_{\textrm{cl}}\right\|_{\infty} < \eta$ is guaranteed to be satisfied with probability $1$, which also means (based on Theorem \ref{th:DEV}) that $\left\|\mathcal{G}_{\textrm{cl}}\right\|_{\infty} < \eta$ for all uncertain parameters $\rho_1,\rho_2,\rho_3,\rho_4$  in the considered uncertainty intervals. }

\begin{figure}
\centering
\includegraphics[]{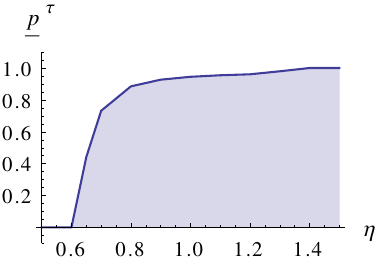}
 \caption{Probabilistic performance analysis: norm bound $\eta$  vs   lower bound $\underline{p}^\tau$ on  $\textrm{Pr}_{\rho}(\|G_\mathrm{cl}\|_{\infty}< \eta)$.} \label{fig:Control}
\end{figure}

\section{Conclusions}
In this paper, we have presented a unified framework for deterministic and probabilistic  analysis of $\mathcal{D}$-stability of uncertain matrices. A family of matrices whose  members have entries
 which vary in an uncertainty   set described by polynomial constraints is considered, and stability regions $\mathcal{D}$ whose complement is described by polynomial constraints  can be handled.  This class of stability sets is quite vast and includes, among others: 
\begin{itemize}
\item the open left half plane and the unit circle of the complex plane, which  allows us to verify stability of continuous- and discrete-time  LTI systems with parametric uncertainty;
\item the imaginary axis, which allows us to compute an upper bound on the H$_{\infty}$-norm of uncertain LTI systems; 
\item the semi-axis of positive real numbers, which allows us to verify robust and probabilistic positive definiteness of a family of real symmetric matrices;
\item the origin of the complex plane,  which allows us to verify robust and probabilistic nonsingularity of uncertain matrices.
\end{itemize} 
 Actually, the approach described in the paper  is widely applicable and it  can undoubtedly be used to  tackle many problems in systems and control theory.

The rationale behind  the method is to formulate a generalized moment optimization problem    which is relaxed through the Lasserre's hierarchy into a sequence of \emph{semidefinite programming} (SDP) problems of finite size. The relaxed problems provide  lower bounds on the minimum probability of a family of matrices to be  $\mathcal{D}$-stable. This is equivalent, in the deterministic realm, to derive sufficient conditions for robust $\mathcal{D}$-stability. It has been observed that, in practice, the level of conservativeness due to the Lasserre's relaxation  is relatively ``small'', and tight solutions are obtained in many cases.

The deterministic and the probabilistic analysis can also be easily combined to handle scenarios where some   parameters are only known to vary within given uncertainty regions, and other parameters are also characterized by probabilistic information (like mean or variance). 
 
Future activities will be  devoted
to extending the ideas underlying the developed method to both   robust and  probabilistic
 control synthesis problems. Furthermore, in order to reduced the computational time required in solving the relaxed SDP problems,  dedicated numerical algorithms   will be developed, thus avoiding the use of general purpose SDP solvers.

%

{
\appendix
\section*{Convergence of the Lasserre's hierarchy}
In this section, we discuss convergence of  the solution $\overline{p}^\tau$ of the SDP relaxed problem \eqref{eqn:probProbsingZvecmomrelax}  to the global optimum $\overline{p}$ of problem  \eqref{eqn:probProbsingZetavec}.
 First, some useful lemmas and  results are given.

\begin{lemma} {\normalfont [\textbf{Putinar's representation of positive polynomials over semialgebraic sets \cite{Putinar93}}]} \label{lemma}

Suppose that the set $\Zvecsupp$ in \eqref{eqn:Zdescv2} is compact and there exists a real-value polynomial $u(z)$ such that $\{z: u(z) \geq 0\}$ is compact and:
\begin{equation}
u(z)=u_0(z)+\sum_{i=1}^{\sizevec{q}}q_i(z)u_i(z), \label{eqn:u}
\end{equation}
where $u_i(z)$ (with $i=0,\ldots,\sizevec{q}$) are all sum-of-squares polynomials. Then, any polynomial $t(z)$ strictly positive on $\Zvecsupp$ can be written as:
\begin{equation*}
t(z)=\sigma_0(z)+\sum_{i=1}^{\sizevec{q}}q_i(z)\sigma_i(z),
\end{equation*}
where $\sigma_i(z)$ (with $i=0,\ldots,\sizevec{q}$) are all  sum-of-squares polynomials (whose degree is not known in advance). 
\end{lemma}

Note that  if the set  $\Zvecsupp$ is included  in the ball $\{z: \|z\|^2 \leq a^2\}$, for $a$ sufficiently large, one way to ensure that the assumptions  in Lemma \ref{lemma} are satisfied is to add in the definition of $\Zvecsupp$ the constraint $q_{\sizevec{q}+1}(z)=a^2- \|z\|^2 \geq 0$ and chose in \eqref{eqn:u} $u_i=0$ ($i=0,\ldots,\sizevec{q}$) and $u_{\sizevec{q}+1}=1$.

\begin{proposition}   \label{prop:dual}
 The dual  of the semi-infinite optimization problem \eqref{eqn:probProbsingZetavec} is:
\begin{subequations} \label{eqn:pdual}
\begin{align} 
t^{*}=\inf_{\nu} & \ \ \nu_1 + \sum_{i=2}^{\sizevec{f}}\nu_i \mu_i \\
s.t. & \ \ \nu_1 + \sum_{i=2}^{\sizevec{f}}\nu_i \tilde{f}_i(z) -h(z) \geq 0 \ \ \forall z \in \Zvecsupp.
\end{align}
\end{subequations} 
From well known results on dual optimization \cite{bertsimas2005optimal}, if  $\mu$ belongs to the interior of the moment space generated by $P_z \in \mathcal{P}_z(\mu)$, then there is no duality gap between problem \eqref{eqn:probProbsingZetavec} and problem \eqref{eqn:pdual}, i.e., 
\begin{equation} \label{eqnp3}
t^{*}=\overline{p}.
\end{equation}
\end{proposition}

\begin{proposition}
Let us write the moment matrix $M_{\tau}(m)$ and the localizing matrices $M_{\tau-\left\lceil \frac{deg(q_j)}{2}  \right\rceil}(q_j m)$ in \eqref{eqn:probProbsingZvecmomrelax} as
\begin{subequations} \label{eqn:MM}
\begin{align}
M_{\tau}(m)=\!\!\!\!& \sum_{\alpha \in \setinteger{2 \tau}{\sizevec{z}}}\!\!\!\!B_\alpha m_\alpha, \\
 M_{\tau-\left\lceil \frac{deg(q_j)}{2}  \right\rceil}(q_j m)=\!\!\!\!& \sum_{\alpha \in \setinteger{2 \tau}{\sizevec{z}}}\!\!\!\!C^{(j)}_\alpha m_\alpha,
\end{align}
\end{subequations}
where $B_\alpha$ and $C^{(j)}_\alpha$ are  symmetric matrices properly defined.

Then, the dual  of the SDP problem \eqref{eqn:probProbsingZvecmomrelax}  is 
 given by:
\begin{subequations} \label{eqn:dd}
{  \begin{align}
& \overline{t}^{\tau}=\inf_{\nu,X \succeq 0,Y^{(j)}\succeq 0}   \ \  \nu^\top \mu \\
&  s.t. \  \nu_1 \!+\!\!\! \sum_{i=2}^{\sizevec{f}}\!\nu_i\bm{\tilde{f}}_{i,\alpha}\!-\!\bm{h}_{\alpha}\!=<\!\! X\!,B_{\alpha}\!\!>\!\!+\!\!\sum_{j=1}^{\sizevec{q}}\!\! <\!\!Y^{(j)}\!\!,C^{(j)}_{\alpha}\!\!>\!,    \alpha \! \in \! \setinteger{2 \tau}{\sizevec{z}}\!,
\end{align}}
\end{subequations}
with $<\!\! X\!,B_{\alpha}\!\!>$ (resp. $ <\!\!Y^{(j)}\!\!,C^{(j)}_{\alpha}\!\!>$) being the trace of  the matrix $XB_{\alpha}$ (resp. $Y^{(j)}C^{(j)}_{\alpha}$).

Obviously, by weak duality:
\begin{equation} \label{eqnp2}
 \overline{t}^{\tau} \geq \overline{p}^{\tau}.
\end{equation}
\end{proposition}

\begin{theorem}
Under the assumptions in Lemma \ref{lemma} and Proposition \ref{prop:dual}, the following convergence condition holds:
$\lim_{\tau \rightarrow \infty} \overline{p}^{\tau}=\overline{p}$. 
\end{theorem}

\begin{proof}
Let $\nu^*$ be the optimal solution of problem \eqref{eqn:pdual}. Thus:
$t^{*}=\nu^{*}_1 + \sum_{i=2}^{\sizevec{f}}\nu^{*}_i \mu_i$,  
and  
$ \nu^*_1 + \sum_{i=2}^{\sizevec{f}}\nu^*_i \tilde{f}_i(z) -h(z) \geq 0 \ \ \forall z \in \Zvecsupp$.
Take $\varepsilon > 0$ arbitrary. Then:
\begin{equation*} \nu^*_1 + \sum_{i=2}^{\sizevec{f}}\nu^*_i \tilde{f}_i(z) -h(z) + \varepsilon > 0 \ \ \forall z \in \Zvecsupp. \end{equation*}
Since the polynomial $ \nu^*_1 + \sum_{i=2}^{\sizevec{f}}\nu^*_i \tilde{f}_i(z) -h(z) + \varepsilon$ is strictly positive  on $\Zvecsupp$, because of Lemma \ref{lemma}, there exist sum-of-squares polynomials $\sigma_j(z)$ ($j=0,\ldots,\sizevec{q}$) such that  
\begin{equation*}
\nu^*_1 + \sum_{i=2}^{\sizevec{f}}\nu^*_i \tilde{f}_i(z) -h(z) + \varepsilon=\sigma_0(z)+\sum_{j=1}^{\sizevec{q}}q_j(z)\sigma_j(z),
\end{equation*}
provided that $\sigma_0(z)$ and $\sigma_j(z)$ ($j=1,\ldots,\sizevec{q}$) have order $2\tau$ and $2\tau-2\left\lceil \frac{deg(q_j)}{2}  \right\rceil$, respectively,  for $\tau$ large enough. 

Let us write the SOS polynomials $\sigma_j(z)$ ($j=0,\ldots,\sizevec{q}$) as
$\sigma_j(z)=\sum_{i=1}^{r_j}\sigma_{ji}(z)^2$, 
and let $\bm{\sigma}_{ji}$ be the vector of coefficients of the polynomial $\sigma_{ji}(z) \in \mathbb{R}_{d_j}[z]$ in the basis $b_{d_j}(z)$, with 
$d_0=\tau$, $d_j=\tau-\left\lceil \frac{deg(q_j)}{2}  \right\rceil, \ \ j=1,\ldots,\sizevec{q}$.
Let us construct the matrices
\begin{equation} \label{eqn:XY}
X=\sum_{i=1}^{r_0}\bm{\sigma}_{0i}\bm{\sigma}'_{0i} \succeq 0, \ \  \ Y^{(j)}=\sum_{i=1}^{r_j}\bm{\sigma}_{ji}\bm{\sigma}'_{ji} \succeq 0.
\end{equation}
For an arbitrary $z\in \mathbb{R}^{\sizevec{z}}$, let us construct the  vector
\begin{equation} \label{eqn:mv}
m=b_{2\tau}(z)=\left[1\  z_1 \ \cdots \  z_{\sizevec{z}} \ z_1^2 \  z_1z_2 \  \cdots \  z_{\sizevec{z}}^{2\tau}\right].
\end{equation} 
Then, with $m$ as in \eqref{eqn:mv}, and $X$ and $Y^{(j)}$ as in \eqref{eqn:XY},  we have
\begin{align} \label{eqn:con4}
&<X,M_{\tau}(m)>+\sum_{j=1}^{\sizevec{q}}<Y^{(j)},M_{\tau-\left\lceil \frac{deg(q_j)}{2}  \right\rceil}(q_jm)>  \nonumber \\
=& \sigma_0(z)+\sum_{j=1}^{\sizevec{q}}q_j(z)\sigma_j(z)=   \displaystyle \nu^*_1 + \sum_{i=2}^{\sizevec{f}}\nu^*_i \tilde{f}_i(z) -h(z) + \varepsilon.
\end{align}
Since $z$ in \eqref{eqn:mv} is arbitrary, condition \eqref{eqn:con4} holds for any  $z\in \mathbb{R}^{\sizevec{z}}$. Thus, by rewriting $M_{\tau}(m)$ and $M_{\tau-\left\lceil \frac{deg(q_j)}{2}  \right\rceil}(q_jm)$ as in \eqref{eqn:MM}, we have:
{\small \begin{equation*}
<\!\! X,B_\alpha\!\!>\!+\!\!\sum_{j=1}^{\sizevec{q}}\!<\!\!Y^{(j)},C^{(j)}_\alpha\!\!>=\! \nu^*_1 \!+\!\!\! \sum_{i=2}^{\sizevec{f}}\!\nu^*_i\bm{\tilde{f}}_{i,\alpha}\!-\!\bm{h}_{\alpha}\!+\!\varepsilon, \  \alpha \! \in \! \setinteger{2 \tau}{\sizevec{z}}.
\end{equation*}}
Thus, $\nu_1=\nu_1^*+\varepsilon$, $\nu_i=\nu_i^*$ ($i=2,\ldots,\sizevec{f}$), and $X$ and $Y^{(j)}$ in \eqref{eqn:XY} are feasible for problem \eqref{eqn:dd}. For these values of $\nu$, $X$ and $Y^{(j)}$,  the cost function in \eqref{eqn:XY} is equal to 
$\nu_1^*+\sum_{i=2}^{\sizevec{f}}\nu_i^*\mu_i+\varepsilon=t^*+\varepsilon$. Therefore,
\begin{equation} \label{eqnp1}
\overline{t}^{\tau} \leq t^*+\varepsilon.
\end{equation} 
By combining eqs. \eqref{eqnp4}, \eqref{eqnp3},  \eqref{eqnp2} and \eqref{eqnp1}, we have:
\begin{equation} \label{eqn5}
t^*=\overline{p} \leq \overline{p}^\tau \leq \overline{t}^{\tau} \leq t^*+\varepsilon.
\end{equation}
 Summarizing, for every $\varepsilon>0$, there exists $\tau$ large enough such that (see \eqref{eqn5}):
$\overline{p} \leq \overline{p}^\tau \leq  \overline{p}+\varepsilon$,
or equivalently, 
$
\lim_{\tau \rightarrow \infty}  \overline{p}^\tau=\overline{p}$.
\end{proof}}

 \section*{Acknowledgment} \vspace{-0.0cm}
The authors would like to thank    Prof. Nicola Guglielmi for the interesting discussions on $\mathcal{D}$-stability analysis and  Prof. Johan L\"ofberg for his suggestions on the implementation of the Lasserre's hierarchy with moment constraints in YALMIP.
%

\bibliographystyle{IEEEtran}
\bibliography{ROB_DIOPH}

\begin{thebibliography}{10}
\providecommand{\url}[1]{#1}
\csname url@samestyle\endcsname
\providecommand{\newblock}{\relax}
\providecommand{\bibinfo}[2]{#2}
\providecommand{\BIBentrySTDinterwordspacing}{\spaceskip=0pt\relax}
\providecommand{\BIBentryALTinterwordstretchfactor}{4}
\providecommand{\BIBentryALTinterwordspacing}{\spaceskip=\fontdimen2\font plus
\BIBentryALTinterwordstretchfactor\fontdimen3\font minus
  \fontdimen4\font\relax}
\providecommand{\BIBforeignlanguage}[2]{{%
\expandafter\ifx\csname l@#1\endcsname\relax
\typeout{** WARNING: IEEEtran.bst: No hyphenation pattern has been}%
\typeout{** loaded for the language `#1'. Using the pattern for}%
\typeout{** the default language instead.}%
\else
\language=\csname l@#1\endcsname
\fi
#2}}
\providecommand{\BIBdecl}{\relax}
\BIBdecl

\bibitem{walley1991}
P.~Walley, \emph{Statistical Reasoning with Imprecise Probabilities}.\hskip 1em
  plus 0.5em minus 0.4em\relax New York: Chapman and Hall, 1991.

\bibitem{benavoli_2011_journal_c}
A.~Benavoli, M.~Zaffalon, and E.~Miranda, ``Robust filtering through coherent
  lower previsions,'' \emph{Automatic Control, IEEE Transactions on}, vol.~56,
  no.~7, pp. 1567 --1581, July 2011.

\bibitem{benavoli2013b}
A.~Benavoli, ``The generalized moment-based filter,'' \emph{Automatic Control,
  IEEE Transactions on}, vol.~58, no.~10, pp. 2642--2647, 2013.

\bibitem{benavoli_2016a}
A.~{Benavoli} and D.~{Piga}, ``{A probabilistic interpretation of
  set-membership filtering: application to polynomial systems through polytopic
  bounding},'' \emph{Automatica}, In press.

\bibitem{2001Ala}
J.~B. Lasserre, ``Global optimization with polynomials and the problem of
  moments,'' \emph{SIAM J. on Optimization}, vol.~11, pp. 796--817, 2001.

\bibitem{PoRo93}
S.~Polijak and J.~Rohn, ``Checking robust non-singularity is {NP}-hard,''
  \emph{Mathematics of Control, Signals, and Systems}, vol.~6, no.~2, pp. 1--9,
  1993.

\bibitem{Ne93}
A.~Nemirovskii, ``Several {NP}-hard problems arising in robust stability
  analysis,'' \emph{Mathematics of Control, Signals, and Systems}, vol.~6,
  no.~2, pp. 99--105, 1993.

\bibitem{gurvits2009np}
L.~Gurvits and A.~Olshevsky, ``On the {NP-hardness} of checking matrix polytope
  stability and continuous-time switching stability,'' \emph{IEEE Transactions
  on Automatic Control}, vol.~54, no.~2, pp. 337--341, 2009.

\bibitem{rohn1989systems}
J.~Rohn, ``Systems of linear interval equations,'' \emph{Linear algebra and its
  applications}, vol. 126, pp. 39--78, 1989.

\bibitem{hertz1992extreme}
D.~Hertz, ``The extreme eigenvalues and stability of real symmetric interval
  matrices,'' \emph{IEEE Transactions on Automatic Control}, vol.~37, no.~4,
  pp. 532--535, 1992.

\bibitem{rohn1994positive}
J.~Rohn, ``Positive definiteness and stability of interval matrices,''
  \emph{{SIAM} Journal on Matrix Analysis and Applications}, vol.~15, no.~1,
  pp. 175--184, 1994.

\bibitem{rohn1996algorithm}
------, ``An algorithm for checking stability of symmetric interval matrices,''
  \emph{IEEE Transactions on Automatic Control}, vol.~41, no.~1, pp. 133--136,
  1996.

\bibitem{deif1990adv}
A.~S. Deif, \emph{Advanced Matrix Theory for for Scientists and
  Engineers}.\hskip 1em plus 0.5em minus 0.4em\relax CRC Press, 1990.

\bibitem{qiu1996bounds}
Z.~Qiu, S.~Chen, and I.~Elishakoff, ``Bounds of eigenvalues for structures with
  an interval description of uncertain-but-non-random parameters,''
  \emph{Chaos, Solitons \& Fractals}, vol.~7, no.~3, pp. 425--434, 1996.

\bibitem{alamo2008new}
T.~Alamo, R.~Tempo, D.~R. Ram{\'\i}rez, and E.~F. Camacho, ``A new vertex
  result for robustness problems with interval matrix uncertainty,''
  \emph{Systems \& Control Letters}, vol.~57, no.~6, pp. 474--481, 2008.

\bibitem{dzhafarov2006stability}
V.~Dzhafarov and T.~B{\"u}y{\"u}kk{\"o}ro{\u{g}}lu, ``On the stability of a
  convex set of matrices,'' \emph{Linear algebra and its applications}, vol.
  414, no.~2, pp. 547--559, 2006.

\bibitem{peaucelle2000new}
D.~Peaucelle, D.~Arzelier, O.~Bachelier, and J.~Bernussou, ``A new robust
  d-stability condition for real convex polytopic uncertainty,'' \emph{Systems
  \& Control Letters}, vol.~40, no.~1, pp. 21--30, 2000.

\bibitem{bliman2004convex}
P.~A. Bliman, ``A convex approach to robust stability for linear systems with
  uncertain scalar parameters,'' \emph{{SIAM} Journal on Control and
  Optimization}, vol.~42, no.~6, pp. 2016--2042, 2004.

\bibitem{leite2003improved}
V.~J.~S. Leite and P.~L.~D. Peres, ``An improved {LMI} condition for robust
  {D}-stability of uncertain polytopic systems,'' in \emph{Proceedings of the
  American Control Conference}, vol.~1, 2003, pp. 833--838.

\bibitem{chesi2005establishing}
G.~Chesi, ``Establishing stability and instability of matrix hypercubes,''
  \emph{Systems \& control letters}, vol.~54, no.~4, pp. 381--388, 2005.

\bibitem{ramos2001less}
D.~C.~W. Ramos and P.~L.~D. Peres, ``A less conservative {LMI} condition for
  the robust stability of discrete-time uncertain systems,'' \emph{Systems \&
  Control Letters}, vol.~43, no.~5, pp. 371--378, 2001.

\bibitem{henrion2003positive}
D.~Henrion, D.~Arzelier, and D.~Peaucelle, ``Positive polynomial matrices and
  improved {LMI} robustness conditions,'' \emph{Automatica}, vol.~39, no.~8,
  pp. 1479--1485, 2003.

\bibitem{ebihara2005robust}
Y.~Ebihara, K.~Maeda, and T.~Hagiwara, ``Robust $\mathcal{D}$-stability
  analysis of uncertain polynomial matrices via polynomial-type multipliers,''
  in \emph{Proceedings of the 16th IFAC World Congress}, Prague, Czech
  Republic, 2005, pp. 191--196.

\bibitem{oliveira2007parameter}
R.~L.~F. Oliveira and P.~L.~D. Peres, ``Parameter-dependent {LMIs in robust
  analysis: characterization of homogeneous polynomially parameter-dependent
  solutions via LMI} relaxations,'' \emph{IEEE Transactions on Automatic
  Control}, vol.~52, no.~7, pp. 1334--1340, 2007.

\bibitem{KiBr2016}
M.~K. Kishida and R.~D. Braatz, ``On the analysis of the eigenvalues of
  uncertain matrices by $\mu$ and $\nu$: Applications to bifurcation avoidance
  and convergence rates,'' \emph{IEEE Trans. on Automatic Control}, vol.~61,
  no.~3, pp. 748--753, 2016.

\bibitem{guglielmi2011fast}
N.~Guglielmi and M.~L. Overton, ``Fast algorithms for the approximation of the
  pseudospectral abscissa and pseudospectral radius of a matrix,'' \emph{{SIAM}
  Journal on Matrix Analysis and Applications}, vol.~32, no.~4, pp. 1166--1192,
  2011.

\bibitem{guglielmi2013low}
N.~Guglielmi and C.~Lubich, ``Low-rank dynamics for computing extremal points
  of real pseudospectra,'' \emph{{SIAM} Journal on Matrix Analysis and
  Applications}, vol.~34, no.~1, pp. 40--66, 2013.

\bibitem{vidyasagar2001probabilistic}
M.~Vidyasagar and V.~D. Blondel, ``Probabilistic solutions to some {NP-hard}
  matrix problems,'' \emph{Automatica}, vol.~37, no.~9, pp. 1397--1405, 2001.

\bibitem{alamo2009randomized}
T.~Alamo, R.~Tempo, and E.~F. Camacho, ``Randomized strategies for
  probabilistic solutions of uncertain feasibility and optimization problems,''
  \emph{IEEE Transactions on Automatic Control}, vol.~54, no.~11, pp.
  2545--2559, 2009.

\bibitem{tempo2012randomized}
R.~Tempo, G.~Calafiore, and F.~Dabbene, \emph{Randomized algorithms for
  analysis and control of uncertain systems: with applications}.\hskip 1em plus
  0.5em minus 0.4em\relax Springer Science \& Business Media, 2012.

\bibitem{Barmish94}
R.~B. Barmish, \emph{New Tools for Robustness of Linear Systems}.\hskip 1em
  plus 0.5em minus 0.4em\relax Macmillan, New York, NY, 1994.

\bibitem{blondel2000survey}
V.~Blondel and J.~N. Tsitsiklis, ``A survey of computational complexity results
  in systems and control,'' \emph{Automatica}, vol.~36, no.~9, pp. 1249--1274,
  2000.

\bibitem{rump2006eigenvalues}
S.~M. Rump, ``Eigenvalues, pseudospectrum and structured perturbations,''
  \emph{Linear algebra and its applications}, vol. 413, no.~2, pp. 567--593,
  2006.

\bibitem{vlassis2014polytopic}
N.~Vlassis and R.~Jungers, ``Polytopic uncertainty for linear systems: New and
  old complexity results,'' \emph{Systems \& Control Letters}, vol.~67, pp.
  9--13, 2014.

\bibitem{freitag2014calculating}
M.~A. Freitag, A.~Spence, and P.~{Van Dooren}, ``Calculating the
  {H$_{\infty}$-norm} using the implicit determinant method,'' \emph{SIAM
  Journal on Matrix Analysis and Applications}, vol.~35, no.~2, pp. 619--635,
  2014.

\bibitem{freitag2014new}
M.~A. Freitag and A.~Spence, ``A new approach for calculating the real
  stability radius,'' \emph{{BIT} Numerical Mathematics}, vol.~54, no.~2, pp.
  381--400, 2014.

\bibitem{petersen2014robust}
I.~R. Petersen and R.~Tempo, ``Robust control of uncertain systems: Classical
  results and recent developments,'' \emph{Automatica}, vol.~50, no.~5, pp.
  1315--1335, 2014.

\bibitem{rostami2015new}
M.~Rostami, ``New algorithms for computing the real structured pseudospectral
  abscissa and the real stability radius of large and sparse matrices,''
  \emph{{SIAM} Journal on Scientific Computing}, vol.~37, no.~5, pp. 447--471,
  2015.

\bibitem{henrion2012inner}
D.~Henrion and J.~B. Lasserre, ``Inner approximations for polynomial matrix
  inequalities and robust stability regions,'' \emph{IEEE Transactions on
  Automatic Control}, vol.~57, no.~6, pp. 1456--1467, 2012.

\bibitem{hess2016semidefinite}
R.~He{\ss}, D.~Henrion, J.~B. Lasserre, and T.~S. Pham, ``Semidefinite
  approximations of the polynomial abscissa,'' \emph{SIAM Journal on Control
  and Optimization}, vol.~54, no.~3, pp. 1633--1656, 2016.

\bibitem{Pi2016}
D.~Piga, ``{Computation of the Structured Singular Value via Moment LMI
  Relaxations},'' \emph{IEEE Transactions on Automatic Control}, vol.~61,
  no.~2, pp. 520--525, 2016.

\bibitem{2009Ala}
M.~Laurent, ``Sums of squares, moment matrices and optimization over
  polynomials,'' \emph{Emerging Applications of Algebraic Geometry, Vol. 149 of
  IMA Volumes in Mathematics and its Applications, M. Putinar and S. Sullivant
  (eds.)}, pp. 157--270, 2009.

\bibitem{lofberg2004yalmip}
J.~L\"ofberg, ``{YALMIP}: A toolbox for modeling and optimization in
  {M}atlab,'' in \emph{IEEE International Symposium on Computer Aided Control
  Systems Design}, Taipei, Taiwan, 2004, pp. 284--289.

\bibitem{1999Ast}
J.~F. Sturm, ``Using {SeDuMi} 1.02, a {MATLAB Toolbox} for optimization over
  symmetric cones,'' \emph{Optim. Methods Software}, vol.~11, no.~12, pp.
  625--653, 1999.

\bibitem{kuznetsov2013elements}
Y.~A. Kuznetsov, \emph{Elements of applied bifurcation theory}.\hskip 1em plus
  0.5em minus 0.4em\relax Springer Science \& Business Media, 2013, vol. 112.

\bibitem{Putinar93}
M.~Putinar, ``Positive polynomials on compact semi-algebraic sets,''
  \emph{Indiana University Mathematics Journal}, vol.~42, pp. 969--984, 1993.

\bibitem{bertsimas2005optimal}
D.~Bertsimas and I.~Popescu, ``Optimal inequalities in probability theory: A
  convex optimization approach,'' \emph{SIAM Journal on Optimization}, vol.~15,
  no.~3, pp. 780--804, 2005.

\end{thebibliography}

\vfill

\end{document}